\documentclass[10pt,a4paper,twoside]{amsart}
\usepackage{graphicx}
\usepackage{a4wide}

\newtheorem{definition}{Definition}[section]
\newtheorem{theorem}{Theorem}[section]
\newtheorem{lemma}{Lemma}[section]
\newtheorem{corollary}{Corollary}[section]

\newtheorem*{maintheorem*}{Main Theorem}
\allowdisplaybreaks
\numberwithin{equation}{section}

\newcommand{\Set}[1]{\left\{#1\right\}}

\newcommand{\norm}[1]{\left\| #1 \right\|}

\newcommand{\Fed}{F_{\eps,\delta}}
\newcommand{\eps}{\varepsilon}

\newcommand{\ed}{{\eps,\delta}}
\newcommand{\ued}{u_\ed}
\newcommand{\Ped}{P_\ed}

\newcommand{\ue}{u_\eps}
\newcommand{\dk}{\delta_k}
\newcommand{\uedk}{u_{\eps,\dk}}
\newcommand{\Fe}{F_\eps}

\newcommand{\uek}{u_{\eps_k}}

\newcommand{\be}{b_{\eps}}
\newcommand{\Pe}{P_\eps}

\newcommand{\Pek}{P_{\eps_k}}

\newcommand{\Pedk}{P_{\eps,\dk}}
\newcommand{\we}{w_\eps}

\newcommand{\pt}{\partial_t}

\newcommand{\px}{\partial_x }
\newcommand{\pxx}{\partial_{xx}^2}
\newcommand{\pxxx}{\partial_{xxx}^3}
\newcommand{\pxxxx}{\partial_{xxxx}^4}

\newcommand{\ptx}{\partial_{tx}^2}

\renewcommand{\i}{\ifmmode\mathit{\mathchar"7010 }\else\char"10 \fi}
\renewcommand{\j}{\ifmmode\mathit{\mathchar"7011 }\else\char"11 \fi}
\newcommand{\R}{\mathbb{R}}
\newcommand{\N}{\mathbb{N}}
\newcommand{\betas}{\beta_{\eps}}
\newcommand{\psit}{\tilde{\psi}}
\newcommand{\supp}{\mathrm{supp}\,}
\newcommand{\Hneg}{H_{\mathrm{loc}}^{-1}}
\newcommand{\CL}{\mathcal{L}}

\newcommand{\ve}{v_\varepsilon}

\newcommand{\oeps}{\omega_{\eps}}
\newcommand{\Oeps}{\Omega_{\eps}}

{%

\begin{enumerate}}%
{\end{enumerate}}

{%

\begin{enumerate}}%
{\end{enumerate}}

\begin{document}\large

\title[The Ostrovsky--Hunter Equation]{Oleinik type estimates \\ for the  Ostrovsky--Hunter Equation}

\date{\today}

\author[G. M. Coclite and L. di Ruvo]{Giuseppe Maria Coclite and Lorenzo di Ruvo}
\address[Giuseppe Maria Coclite and Lorenzo di Ruvo]
{\newline Department of Mathematics,   University of Bari, via E. Orabona 4, 70125 Bari,   Italy}
\email[]{giuseppemaria.coclite@uniba.it, lorenzo.diruvo@uniba.it}
\urladdr{http://www.dm.uniba.it/Members/coclitegm/}

\keywords{Existence, Uniqueness, Stability, Entropy solutions, Conservation Laws,
Ostrovsky-Hunter equation, Cauchy problems, Oleinik estimate.}

\subjclass[2000]{35G25, 35L65, 35L05, 35A05}


\begin{abstract}
The Ostrovsky-Hunter equation provides a model for small-amplitude long waves in
a rotating fluid of finite depth. It is a nonlinear evolution equation. In this paper we study the well-posedness for the Cauchy problem associated to this equation
within a class of bounded discontinuous solutions.
We show that we can replace the Kruzkov-type entropy inequalities by an Oleinik-type estimate and  prove uniqueness
via a nonlocal adjoint problem. An implication is that a shock wave in an entropy weak solution to the
Ostrovsky-Hunter equation is admissible only if it jumps down in value (like the inviscid Burgers equation).
\end{abstract}

\maketitle

\section{Introduction}
\label{sec:intro}

Our aim is to investigate the  well-posedness in classes of discontinuous functions for the equation
\begin{equation}
\label{eq:OH}
\px(\pt u+\px f(u))=\gamma u,\qquad t>0,\>x\in\R.
\end{equation}
We are interested in the Cauchy problem for this equation, so we augment \eqref{eq:OH} with the initial condition
\begin{equation}
\label{eq:init}
u(0,x)=u_0(x), \qquad x\in\R,
\end{equation}
on which we assume that
\begin{equation}
\label{eq:assinit}
u_0\in L^1(\R)\cap L^{\infty}(\R),\quad\int_{\R}u_{0}(x)dx=0.
\end{equation}
On the function
\begin{equation}
\label{eq:def-di-P0}
P_{0}(x)=\int_{-\infty}^{x} u_{0}(y)dy, \quad x\in\R,
\end{equation}
we assume that
\begin{equation}
\label{eq:L-2P0}
\begin{split}
\norm{P_0}^2_{L^2(\R)}&=\int_{\R}\left(\int_{-\infty}^{x}u_{0}(y)dy\right)^2dx <\infty,\\
\int_{\R}P_0(x)dx&= \int_{\R}\left(\int_{-\infty}^{x}u_{0}(y)dy\right)dx=0.
\end{split}
\end{equation}
The flux $f$ is assumed to be smooth, Lipschitz continuous and {\em strictly convex,} namely:
\begin{equation}
\label{eq:assflux1}
f\in C^2(\R),\qquad f''\ge C_{0},\qquad \vert f'(u) \vert \le C_{0} \vert u \vert, \quad u\in\R,
\end{equation}
for some a positive constant $C_0$.

The equation \eqref{eq:OH} is the limit of no high-frequency dispersion ($\beta=0$) of the non-linear evolution equation
\begin{equation}
\label{eq:OHbeta}
\px(\pt u+u\px u-\beta \pxxx u)=\gamma u,
\end{equation}
that was derived by Ostrovsky \cite{O} to model small-amplitude long waves in a rotating fluid of a finite depth.
It generalizes the Korteweg-deVries equation (that corresponds to $\gamma=0$) by the additional term induced by the Coriolis force.
Mathematical properties of the Ostrovsky equation \eqref{eq:OHbeta} were studied recently in many details
 including the local and global well-posedness in energy space
\cite{GP, GL, LM, LPS, LV, T}, stability of solitary waves \cite{LL, L, LV:JDE}, convergence of solutions in the
limit, $\gamma\to0$, of the Korteweg-deVries equation \cite{LL:07, LV:JDE}, and convergence of solutions in the
limit,  $\beta\to0$, of no high-frequency dispersion \cite{Cd2}.
\eqref{eq:OH}  is deduced considering two asymptotic expansions of the shallow water equations, first with respect to the rotation frequency and then with respect to the amplitude of the waves (see \cite{dR, HT}). It is known under different names such as the reduced Ostrovsky equation \cite{P, S}, the
Ostrovsky-Hunter equation \cite{B}, the short-wave equation \cite{H}, and the Vakhnenko equation
\cite{MPV, PV}.

Integrating \eqref{eq:OH} on $(-\infty,x)$ we gain the integro-differential formulation of problem \eqref{eq:OH}, and \eqref{eq:init} (see \cite{LV})
\begin{equation}
\label{eq:OHw-u}
\begin{cases}
\pt u+u\px u =\gamma \int^x_{-\infty} u(t,y) dy,&\qquad t>0, \ x\in\R,\\
u(0,x)=u_0(x), &\qquad x\in\R,
\end{cases}
\end{equation}
that is equivalent to
\begin{equation}
\label{eq:OHw}
\begin{cases}
\pt u+u\px u=\gamma P,&\qquad t>0, \ x\in\R ,\\
\px P=u,&\qquad t>0, \ x>\in\R,\\
u(0,x)=u_0(x), &\qquad x\in\R.
\end{cases}
\end{equation}

One of the main issues in the analysis of \eqref{eq:OH} is that the equation is not preserving the $L^1$ norm, the unique useful conserved quantities are
\begin{equation*}
t\longmapsto\int_\R u(t,x)dx,\qquad t\longmapsto\int_\R u^2(t,x)dx.
\end{equation*}
As a consequence the nonlocal source term $P$ and the solution $u$ are a priori only locally bounded and not summable with respect to x.
A complete analysis of the well-posedness in that framework can be found in \cite{CdK, dR} under the additional condition
\begin{equation*}
P(t,0)=0,
\end{equation*}
that is natural in the reformulation of the boundary value problems for \eqref{eq:OH}, see \cite{Cd, Cd1}.
The equation analyzed in \cite{CdK, dR} is
\begin{equation}
\label{eq:OHw-u-0}
\pt u+u\px u =\gamma \int^x_{0} u(t,y) dy,\qquad t>0, \ x\in\R
\end{equation}
and not the one in \eqref{eq:OHw-u}.
The two reformulations  \eqref{eq:OHw-u} and \eqref{eq:OHw-u-0} of \eqref{eq:OH} are not equivalent.
Therefore, the well-posedness result of \cite{CdK, dR} do not apply also to \eqref{eq:OHw-u}.
Finally, the Kruzkov doubling of variables works for \eqref{eq:OHw-u-0} but does not for \eqref{eq:OHw-u}.

We are interested in the bounded solutions of \eqref{eq:OH} (the ones of \cite{CdK, dR} are only locally bounded).
Indeed, we have \eqref{eq:L-2P0}, that is an assumption on the decay at infinity of the initial condition $u_0$.
The subquadratic assumption \eqref{eq:assflux1} together with \eqref{eq:L-2P0} guarantees the boundedness of the solutions.
Moreover, the convexity of the flux $f$ is necessary for two reasons. It allows us to use a compensated compactness argument for the existence of weak solutions.
In addition it gives an Oleinik type etimate. We will  show that we can replace the Kruzkov-type entropy inequalities used in \cite{CdK, dR}  by
an Oleinik-type estimate and to prove uniqueness via a nonlocal adjoint problem. An implication
is that a shock wave in an entropy weak solution to the Ostrovsky-Hunter equation is admissible
only if it jumps down in value (like the inviscid Burgers equation).

\begin{definition}
\label{def:sol}
We say that $u\in  L^{\infty}((0,T)\times\R)$ is an entropy solution of the initial
value problem \eqref{eq:OH}, and  \eqref{eq:init} if
\begin{itemize}
\item[$i$)] $u$ is a distributional solution of \eqref{eq:OHw-u} or equivalently of \eqref{eq:OHw};
\item[$ii$)] for every convex function $\eta\in  C^2(\R)$ the
entropy inequality
\begin{equation}
\label{eq:OHentropy}
\pt \eta(u)+ \px q(u)-\gamma\eta'(u) P\le 0, \qquad     q(u)=\int^u f'(\xi) \eta'(\xi)\, d\xi,
\end{equation}
holds in the sense of distributions in $(0,\infty)\times\R$.
\end{itemize}
\end{definition}

The main result of this paper is the following theorem.

\begin{theorem}
\label{th:main}
Assume \eqref{eq:assinit}, \eqref{eq:def-di-P0}, \eqref{eq:L-2P0} and \eqref{eq:assflux1}.
The initial value problem \eqref{eq:OH} and \eqref{eq:init}, possesses
an unique entropy solution $u$ in the sense of Definition \ref{def:sol}.\\
Moreover, the following statements are equivalent:
\begin{itemize}
\item[$i$)] $u$ is an entropy solution of \eqref{eq:OHw-u} or \eqref{eq:OHw} in the sense of Definition \eqref{def:sol};
\item[$ii$)]$u$ is a distributional solution of \eqref{eq:OHw-u} or \eqref{eq:OHw} such that for every $T>0$, there exists $C(T)>0$ such that
\begin{equation}
\label{eq:ole1}
\frac{u(t,x)-u(t,y)}{x-y}\le C(T)\left(\frac{1}{t}+1\right),
\end{equation}
for every $0< t <T$, $x\ne y$.
\end{itemize}
\end{theorem}
The paper is organized in three sections. In Section \ref{sec:vv}, we prove the wellposednees of the approximate solutions of \eqref{eq:OHw-u}, or \eqref{eq:OHw}. In Section \ref{sec:Es}, we prove the existence of the entropy solutions for \eqref{eq:OHw-u}, or \eqref{eq:OHw}, while in Section \ref{sec:olei-uni}, we prove an Oleinik type estimate and Theorem \ref{th:main}.

\section{Wellposedness of the approximate problem}\label{sec:vv}

To prove the existence of entropy solution for \eqref{eq:OHw-u}, or \eqref{eq:OHw}, we analyze the following mixed problem
\begin{equation}
\label{eq:OHepsw}
\begin{cases}
\pt \ue+\px f(\ue)=\gamma\Pe+ \eps\pxx\ue,&\quad t>0,\ x\in\R,\\
\px\Pe=\ue,&\quad t>0,\ x\in\R,\\
\ue(0,x)=u_{\eps,0}(x),&\quad x\in\R,
\end{cases}
\end{equation}
where $\eps>0$ is a small fixed number.
Clearly, \eqref{eq:OHepsw} is equivalent to the integro-differential problem
\begin{equation}
\label{eq:OHepswint}
\begin{cases}
\pt \ue+\px f(\ue)=\gamma\int_{-\infty}^x \ue (t,y)dy+ \eps\pxx\ue,&\quad t>0,\ x\in\R,\\
\ue(0,x)=u_{\eps,0}(x),&\quad x\in\R.
\end{cases}
\end{equation}
This section is devoted to the wellposedness of \eqref{eq:OHepsw}, or \eqref{eq:OHepswint}.
We assume that
\begin{equation}
\label{eq:assinit1}
u_{\eps,0}\in L^1(\R)\cap L^{\infty}(\R)\cap C^{\infty}(\R), \quad\int_{\R}u_{\eps,0}(x)dx=0.
\end{equation}
while on the function
\begin{equation}
\label{eq:def-di-Peps0}
P_{\eps,0}(x)=\int_{-\infty}^{x} u_{\eps,0}(y)dy, \quad x\in\R,
\end{equation}
we assume that
\begin{equation}
\label{eq:L-2Peps0}
\begin{split}
\norm{P_{\eps,0}}^2_{L^2(\R)}&=\int_{\R}\left(\int_{-\infty}^{x}u_{\eps,0}(y)dy\right)^2dx <\infty,\\
\int_{\R}P_{\eps,0}(x)dx&= \int_{\R}\left(\int_{-\infty}^{x}u_{\eps,0}(y)dy\right)dx=0.
\end{split}
\end{equation}
Fix $0<\delta <1$, and let $\ued=\ued (t,x)$ be the unique
classical solution of the following mixed problem \cite{CHK:ParEll}:
\begin{equation}
\label{eq:OHepsw1}
\begin{cases}
\pt \ued+\px f(\ued)=\gamma\Ped +\eps\pxx\ued,&\quad t>0,\ x\in\R,\\
-\delta\pxx\Ped+\px\Ped=\ued,&\quad t>0,\ x\in\R,\\
\ued(0,x)=u_{\eps,\delta,0}(x),&\quad x\in\R,
\end{cases}
\end{equation}
where $u_{\eps,\delta,0}$ is a $C^\infty$ approximation of $u_{\eps,0}$ such that
\begin{equation}
\label{eq:u0eps}
\begin{split}
&\norm{u_{\eps,\delta,0}}_{L^2(\R)}\le \norm{u_{\eps,0}}_{L^2(\R)}, \quad \norm{u_{\eps,\delta,0}}_{L^{\infty}(\R)}\le \norm{u_{\eps,0}}_{L^{\infty}(\R)},\\
&\eps\norm{\px u_{\eps,\delta,0}}_{L^2(\R)}\le C_{0},\quad   \eps\norm{\pxx u_{\eps,\delta,0}}_{L^2(\R)}\le C_{0}\\
&\norm{P_{\eps,\delta,0}}_{L^2(\R)}\le \norm{P_{\eps,0}}_{L^2(\R)},\quad\delta\norm{\px P_{\eps,\delta,0}}_{L^2(\R)}\le C_{0},
\end{split}
\end{equation}
and $C_0$ is a constant independent on $\delta$, but dependent on $\eps$.

The main result of this section is the following theorem:
\begin{theorem}\label{th:wellp}
Let $T>0$. Assume \eqref{eq:assflux1}, \eqref{eq:assinit1}, \eqref{eq:L-2Peps0} and \eqref{eq:def-di-Peps0}. Then there exist
\begin{equation}
\label{eq:uePe}
\begin{split}
\ue&\in L^{\infty}((0,T)\times\R)\cap C((0,T);H^\ell(\R)),\quad \ell > 2, \\
\Pe&\in L^{\infty}((0,T)\times\R)\cap L^{2}((0,T)\times\R),
\end{split}
\end{equation}
where $\ue$ is a unique classic solution of the Cauchy problem of \eqref{eq:OHepsw}.\\
Moreover, if $\ue$ and $\ve$ are two solutions of \eqref{eq:OHepsw}, the following inequality holds
\begin{equation}
\label{eq:l2-stability}
\norm{\ue(t,\cdot)-\ve(t,\cdot)}_{L^2(\R)}\le e^{C_{\eps}(T)t}\norm{u_{\eps,0}-v_{\eps,0}}_{L^2(\R)},
\end{equation}
for some suitable $C_{\eps}(T)>0$, and every $0\le t\le T$.
\end{theorem}
We begin by proving some a priori estimates on $\ued$ and $\Ped$, denoting with $C_0$ the constants which depend on the initial data, and $C(T)$ the constants which depend also on $T$.

\begin{lemma}
\label{lm:cns1}
For each $t\in (0,\infty)$,
\begin{equation}
\label{eq:P-pxP-intfy1}
\Ped(t,\infty)=\px\Ped(t,-\infty)=\px\Pe(t,\infty)=0.
\end{equation}
Moreover,
\begin{equation}
\label{eq:equ-L2-stima1}
\delta^2\norm{\pxx\Ped(t,\cdot)}^2_{L^2(\R)}+ \norm{\px\Ped(t,\cdot)}^2_{L^2(\R)}=\norm{\ued(t,\cdot)}^2_{L^2(\R)}.
\end{equation}
\end{lemma}
\begin{proof}
We begin by proving that \eqref{eq:P-pxP-intfy1} holds true.

Differentiating the first equation of \eqref{eq:OHepsw1} with respect to $x$, we have
\begin{equation}
\label{eq:pxu}
\px(\pt \ued+\px f(\ued)-\eps\pxx\ued)=\gamma\px\Ped.
\end{equation}
For the the smoothness of $\ued$, it follows from \eqref{eq:OHepsw1} and \eqref{eq:pxu} that
\begin{align*}
&\lim_{x\to\infty}(\pt \ued +\px f(\ued)-\eps\pxx\ued)=\gamma\Ped(t,\infty)=0,\\
&\lim_{x\to -\infty}\px(\pt \ued+\px f(\ued)-\eps\pxx\ued)=\gamma\px\Ped(t,-\infty)=0,\\
&\lim_{x\to\infty}\px(\pt \ued+ \px f(\ued)-\eps\pxx\ued)=\gamma\px\Ped(t,\infty)=0,
\end{align*}
which gives \eqref{eq:P-pxP-intfy1}.

Let us show that \eqref{eq:equ-L2-stima1} holds true.
Squaring the equation for $\Pe$ in \eqref{eq:OHepsw1}, we get
\begin{equation*}
\delta^2(\pxx\Ped)^2+(\px\Ped)^2 - \delta\px((\px\Ped)^2)=\ued^2.
\end{equation*}
Therefore, \eqref{eq:equ-L2-stima1} follows from  \eqref{eq:P-pxP-intfy1} and an integration on $\R$.
\end{proof}

\begin{lemma}
\label{lm:2}
For each $t\in(0,\infty)$,
\begin{align}
\label{eq:L-infty-Px}
\sqrt{\delta}\norm{\px\Ped(t, \cdot)}_{L^{\infty}(\R)}&\le \norm{\ued(t,\cdot)}_{L^2(\R)},\\
\label{eq:uP}
\int_{\R}\ued(t,x)\Ped(t,x) dx&\le \norm{\ued(t,\cdot)}^2_{L^2(\R)}.
\end{align}
\end{lemma}

\begin{proof}
We begin by proving that \eqref{eq:L-infty-Px} holds true.\\
Observe that
\begin{equation*}
0\le (-\delta \pxx\Pe + \px\Pe)^2= \delta^2(\pxx\Pe)^2 +(\px\Pe)^2 - \delta\px((\px\Pe)^2),
\end{equation*}
that is,
\begin{equation}
\label{eq:equa-pxP}
\delta\px((\px\Ped)^2)\le \delta^2(\pxx\Ped)^2 +(\px\Ped)^2.
\end{equation}
Integrating \eqref{eq:equa-pxP} on $(-\infty, x)$, we have
\begin{equation}
\label{eq:equa-pxP1}
\begin{split}
\delta(\px\Ped)^2  &\le \delta^2\int_{-\infty}^{x}(\pxx\Ped)^2 dx +\int_{-\infty}^{x}(\px\Ped)^2 dx\\
&\le \delta^2\int_{\R}(\pxx\Ped)^2 dx +\int_{\R}(\px\Ped)^2 dx.
\end{split}
\end{equation}
It follows from \eqref{eq:equ-L2-stima1} and \eqref{eq:equa-pxP1} that
\begin{equation*}
\delta(\px\Ped)^2\le \delta^2\int_{\R}(\pxx\Ped)^2 dx +\int_{\R}(\px\Ped)^2 dx= \norm{\ued(t,\cdot)}^2_{L^2(\R)}.
\end{equation*}
Therefore,
\begin{equation*}
\sqrt{\delta}\vert \px\Ped(t,x)\vert \le \norm{\ued(t,\cdot)}_{L^2(\R)},
\end{equation*}
which gives \eqref{eq:L-infty-Px}.

Finally, we prove \eqref{eq:uP}.
Multiplying by $\Ped$ the equation for $\Ped$ in \eqref{eq:OHepsw1}, we get
\begin{equation*}
-\delta\Ped\pxx\Ped + \Ped\px\Ped= \ued\Ped.
\end{equation*}
An integration on $\R$ and \eqref{eq:P-pxP-intfy1} give
\begin{align*}
\int_{\R}\ued\Ped dx=&\frac{1}{2}\int_{\R}\px(\Pe)^2 dx - \delta \int_{\R}\Ped\pxx\Ped dx\\
=&-\delta\int_{\R}\Ped\pxx\Ped dx=\delta\int_{\R}(\px\Ped)^2 dx,
\end{align*}
that is
\begin{equation*}
\int_{\R}\ued\Ped dx = \delta\int_{\R}(\px\Ped)^2 dx.
\end{equation*}
Since $0<\delta <1$, for \eqref{eq:equ-L2-stima1}, we have \eqref{eq:uP}.
\end{proof}

\begin{lemma}
\label{lm:l2-u}
For each $t\in(0,\infty)$, the following inequality holds
\begin{equation}
\label{eq:l2-u}
\norm{\ued(t,\cdot)}_{L^2(\R)}^2+2\eps e^{2\gamma t}
\int_0^t e^{-2\gamma s}\norm{\px \ued(s,\cdot)}^2_{L^2(\R)}ds\le e^{2\gamma t}\norm{u_{\eps,0}}^2_{L^2(\R)}.
\end{equation}
In particular, we have
\begin{equation}
\label{eq:h2-P}
\norm{\px \Ped(t,\cdot)}_{L^2(\R)},\,\delta \norm{\pxx \Ped(t,\cdot)}_{L^2(\R)},\,\sqrt{\delta}\norm{\px\Ped(t,\cdot)}_{L^\infty(\R)}\le  e^{\gamma t}\norm{u_{\eps,0}}_{L^2(\R)},
\end{equation}
\end{lemma}

\begin{proof}
Due to \eqref{eq:OHepsw1} and \eqref{eq:uP},
\begin{align*}
\frac{d}{dt}\int_{\R} \ued^2dx=&2\int_{\R} \ued\pt\ued dx\\
=&2\eps\int_{\R}\ued\pxx\ued dx-2\int_{\R} \ued f'(\ued)\px\ued dx+2\gamma\int_{\R} \ued\Ped dx\\
\le&-2\eps\int_{\R}(\px\ued)^2 dx+2\gamma\norm{\ued(t,\cdot)}_{L^2(\R)}^2.
\end{align*}
The Gronwall Lemma and  \eqref{eq:u0eps} give \eqref{eq:l2-u}.

Finally, \eqref{eq:h2-P} follows from \eqref{eq:equ-L2-stima1}, \eqref{eq:L-infty-Px} and \eqref{eq:l2-u}.
\end{proof}

\begin{lemma}
\label{lm:p8}
For each $t\ge 0$, we have that
\begin{align}
\label{eq:intp-infty}
\int_{0}^{-\infty}\Ped(t,x)dx&=a_{\eps,\delta}(t), \\
\label{eq:int+infty}
\int_{0}^{\infty}\Ped(t,x)dx&=a_{\eps,\delta}(t),
\end{align}
where
\begin{equation}
a_{\eps,\delta}(t)=\frac{\delta}{\gamma}\ptx\Ped(t,0)- \frac{1}{\gamma}\pt \Ped(t,0)+ \frac{1}{\gamma}f(0)-\frac{1}{\gamma}f(\ued(t,0)) + \frac{\eps}{\gamma}\px\ued(t,0).
\end{equation}
In particular,
\begin{equation}
\label{eq:Pmedianulla}
\int_{\R}\Ped(t,x)dx=0, \quad t\geq 0.
\end{equation}
\end{lemma}
\begin{proof}
We begin by observing that, integrating the second equation of \eqref{eq:OHepsw1} on $(0,x)$, we have that
\begin{equation}
\label{eq:P-in-0}
\int_{0}^{x} \ued(t,y)dy = \Ped(t,x)-\Ped(t,0)-\delta\px\Ped(t,x)+\delta\px\Ped(t,0).
\end{equation}
It follows from \eqref{eq:P-pxP-intfy1} that
\begin{equation}
\label{eq:lim-int-u}
\lim_{x\to -\infty}\int_{0}^{x} \ued(t,y)dy=\int_{0}^{-\infty}\ued(t,x)dx =  \delta\px\Ped(t,0) -  \Ped(t,0).
\end{equation}
Differentiating \eqref{eq:lim-int-u} with respect to $t$, we get
\begin{equation}
\label{eq:lim-int-u-in-t}
\frac{d}{dt}\int_{0}^{-\infty}\ued(t,x)dx= \int_{0}^{-\infty}\pt\ued(t,x)dx=\delta\ptx\Ped(t,0) - \pt \Ped(t,0).
\end{equation}
Integrating the first equation \eqref{eq:OHepsw} on $(0,x)$, we obtain that
\begin{equation}
\begin{split}
\label{eq:int-1-eq}
\int_{0}^{x}\pt\ued(t,y) dy &+ f(\ued(t,x))-f(\ued(t,0))\\
&-\eps\px\ued(t,x)+ \eps\px\ued(t,0)=\gamma\int_{0}^{x}\Ped(t,y)dy.
\end{split}
\end{equation}
Being $\ued$ a smooth solution of \eqref{eq:OHepsw}, we get
\begin{equation}
\label{eq:500}
\lim_{x\to-\infty}\Big( f(\ued(t,x))-\eps\px\ued(t,x)\Big)=f(0).
\end{equation}
Sending $x\to -\infty$ in \eqref{eq:int-1-eq}, for \eqref{eq:lim-int-u-in-t} and \eqref{eq:500}, we have
\begin{align*}
\gamma\int_{0}^{-\infty}\Ped(t,x)dx= &\delta\ptx\Ped(t,0) - \pt \Ped(t,0)\\
&+f(0)-f(\ued(t,0)) + \eps \px\ued(t,0),
\end{align*}
which gives \eqref{eq:intp-infty}.

Let us show that \eqref{eq:int+infty} holds true. We begin by observing that, for \eqref{eq:P-pxP-intfy1} and \eqref{eq:P-in-0},
\begin{equation*}
\int_{0}^{\infty}\ued(t,x)dx =  \delta\px\Ped(t,0) -  \Ped(t,0).
\end{equation*}
Therefore,
\begin{equation}
\label{eq:lim-int-u-1}
\lim_{x\to \infty}\int_{0}^{x}\pt \ued(t,y)dy=\int_{0}^{\infty}\pt\ued(t,x)dx =  \delta\ptx\Ped(t,0) -  \pt\Ped(t,0).
\end{equation}
Again by the regularity of $\ued$,
\begin{equation}
\label{eq:510}
\lim_{x\to\infty}\Big( f(\ued(t,x))-\eps\px\ued(t,x)\Big)=f(0).
\end{equation}
It follows from \eqref{eq:int-1-eq}, \eqref{eq:lim-int-u-1} and \eqref{eq:510} that
\begin{align*}
\gamma\int_{0}^{\infty}\Ped(t,x)dx= &\delta\ptx\Ped(t,0) - \pt \Ped(t,0)\\
&+f(0)-f(\ued(t,0)) + \eps\px\ued(t,0),
\end{align*}
which gives \eqref{eq:int+infty}.

Finally, we prove \eqref{eq:Pmedianulla}. It follows from \eqref{eq:intp-infty} that
\begin{equation*}
\int_{-\infty}^{0}\Ped(t,x)dx = -a_{\eps,\delta}(t).
\end{equation*}
Therefore, for \eqref{eq:int+infty},
\begin{align*}
\int_{-\infty}^{0}\Ped(t,x)dx+\int_{0}^{\infty}\Ped(t,x)=\int_{\R} \Ped(t,x)dx =-a_{\eps,\delta}(t)+a_{\eps,\delta}(t)=0,
\end{align*}
that is \eqref{eq:Pmedianulla}.
\end{proof}

Lemma \ref{lm:p8} says that $\Ped(t,x)$  is integrable at $\pm\infty$.
Therefore, for each $t\ge 0$, we can consider the following function
\begin{equation}
\label{eq:F1}
\Fed(t,x)=\int_{-\infty}^{x}\Ped(t,y)dy.
\end{equation}

\begin{lemma}
\label{lm:P-infty}
Let $T>0$. There exists a function $C(T)>0$, independent on $\delta$, such that
\begin{align}
\label{eq:P-infty}
\norm{\Ped}_{L^{\infty}(I_{T,1})}&\le C(T),\\
\label{eq:l2P}
\norm{\Ped(t,\cdot)}_{L^2(\R)}&\le C(T),\\
\label{eq:l2pxP}
\delta\norm{\px\Ped(t,\cdot)}_{L^2(\R)}&\le C(T),
\end{align}
where
\begin{equation}
\label{eq:defI}
I_{T,1}=(0,T)\times\R.
\end{equation}
In particular, we have
\begin{equation}
\label{eq:pt-px-P}
\delta\left\vert\int_{0}^{t}\!\!\!\int_{\R}\Ped\ptx\Ped ds dx\right\vert\le C(T), \quad 0<t<T.
\end{equation}
\end{lemma}
\begin{proof}
Integrating the second equation of \eqref{eq:OHepsw} on $(-\infty, x)$, for \eqref{eq:P-pxP-intfy1},  we have that
\begin{equation}
\label{eq:1550}
\int_{-\infty}^{x} \ued(t,y)dy=\Ped(t,x) -\delta\px\Ped(t,x).
\end{equation}
Differentiating \eqref{eq:1550} with respect to $t$, we get
\begin{equation}
\label{eq:1551}
\frac{d}{dt}\int_{-\infty}^{x} \ued(t,y)dy=\int_{-\infty}^{x}\pt \ued(t,y)dy=\pt\Ped(t,x) -\delta\ptx\Ped(t,x).
\end{equation}
It follows from an integration of the first equation of \eqref{eq:OHepsw1} on $(-\infty, x)$ and \eqref{eq:F1} that
\begin{equation}
\label{eq:1552}
\int_{-\infty}^{x}\pt\ued(t,y)dy + f(\ued(t,x)) - \eps\px\ued(t,x)=\gamma\Fed(t,x).
\end{equation}
Due to \eqref{eq:1551} and \eqref{eq:1552}, we have
\begin{equation}
\label{eq:1554}
\pt\Ped(t,x)-\delta\ptx\Ped(t,x) =\gamma\Fed(t,x)- f(\ued(t,x)) +\eps\px\ued(t,x).
\end{equation}
Multiplying \eqref{eq:1554} by $\Ped - \delta\px\Ped$, we have
\begin{equation}
\label{eq:1555}
\begin{split}
(\pt\Ped-\delta\ptx\Ped)(\Ped - \delta\px\Ped)= &\gamma\Fed(\Ped - \delta\px\Ped)\\
&- f(\ued)(\Ped - \delta\px\Ped)\\
&+\eps\px\ued(\Ped - \delta\px\Ped).
\end{split}
\end{equation}
Integrating \eqref{eq:1555} on $(0,x)$, we have
\begin{equation}
\label{eq:1222}
\begin{split}
\int_{0}^{x}\pt\Ped\Ped dy&-\delta\int_{0}^{x} \pt\Ped\px\Ped dy\\
&-\delta \int_{0}^{x}\Ped\ptx\Ped dy +\delta^2\int_{0}^{x}\ptx\Ped\px\Ped dy\\
=& \gamma\int_{0}^{x}\Fed\Ped dy - \gamma\delta\int_{0}^{x} \Fed\px\Ped dy\\
&-\int_{0}^{x}f(\ued)\Ped dy + \delta \int_{0}^{x}f(\ued)\px\Ped dy\\
&+\eps\int_{0}^{x}\px\ued\Ped dy - \eps\delta\int_{0}^{x}\px\ued\px\Ped dy.
\end{split}
\end{equation}
We observe that
\begin{equation}
\label{eq:int-by-part}
-\delta \int_{0}^{x}\px\Ped\pt\Ped dy=-\delta\Ped\pt\Ped +\delta\Ped(t,0)\pt\Ped(t,0) + \delta\int_{0}^{x}\Ped\ptx\Ped dy.
\end{equation}
Therefore, \eqref{eq:1222} and \eqref{eq:int-by-part} give
\begin{equation}
\begin{split}
\label{eq:1223}
\int_{0}^{x}\pt\Ped\Ped dy&+ \delta^2\int_{0}^{x}\ptx\Ped\px\Ped dy\\
=& \delta\Ped\pt\Ped  -\delta\Ped(t,0)\pt\Ped(t,0)  + \gamma\int_{0}^{x}\Fed\Ped dy \\
&- \gamma\delta\int_{0}^{x} \Fed\px\Ped dy-\int_{0}^{x}f(\ued)\Ped dy + \delta \int_{0}^{x}f(\ued)\px\Ped dy\\
&+\eps\int_{0}^{x}\px\ued\Ped dy - \eps\delta\int_{0}^{x}\px\ued\px\Ped dy.
\end{split}
\end{equation}
Sending $x\to -\infty$, for \eqref{eq:P-pxP-intfy1}, we get
\begin{equation}
\label{eq:0012}
\begin{split}
\int_{0}^{-\infty}\pt\Ped\Ped dy&+ \delta^2\int_{0}^{-\infty}\ptx\Ped\px\Ped dy\\
=& -\delta\Ped(t,0)\pt\Ped(t,0) + \gamma\int_{0}^{-\infty}\Fed\Ped dy \\
&- \gamma\delta\int_{0}^{-\infty} \Fed\px\Ped dy-\int_{0}^{-\infty}f(\ued)\Ped dy \\
&+ \delta \int_{0}^{-\infty}f(\ued)\px\Ped dy+\eps\int_{0}^{-\infty}\px\ued\Ped dy\\
& - \eps\delta\int_{0}^{-\infty}\px\ued\px\Ped dy,
\end{split}
\end{equation}
while sending $x\to\infty$,
\begin{equation}
\label{eq:0013}
\begin{split}
\int_{0}^{\infty}\pt\Ped\Ped dy&+ \delta^2\int_{0}^{\infty}\ptx\Ped\px\Ped dy\\
=& -\delta\Ped(t,0)\pt\Ped(t,0) + \gamma\int_{0}^{\infty}\Fed\Ped dy- \gamma\delta\int_{0}^{\infty} \Fed\px\Ped dy \\
&-\int_{0}^{\infty}f(\ued)\Ped dy + \delta \int_{0}^{\infty}f(\ued)\px\Ped dy\\
&+\eps\int_{0}^{\infty}\px\ued\Ped dy - \eps\delta\int_{0}^{\infty}\px\ued\px\Ped dy.
\end{split}
\end{equation}
Since
\begin{align*}
\int_{\R}\Ped\pt\Ped dx &=\frac{1}{2}\frac{d}{dt}\int_{\R}\Ped^2dx,\\
\delta^2\int_{\R}\ptx\Ped\px\Ped dx &= \frac{\delta^2}{2}\frac{d}{dt}\int_{\R}(\px\Ped)^2dx,
\end{align*}
it follows from \eqref{eq:0012} and \eqref{eq:0013} that
\begin{equation}
\label{eq:12312}
\begin{split}
\frac{1}{2}\frac{d}{dt}\int_{\R}\Ped^2dx&+\frac{\delta^2}{2}\frac{d}{dt}\int_{\R}(\px\Ped)^2dx\\
=& \gamma\int_{\R}\Fed\Ped dx - \gamma\delta\int_{\R} \Fed\px\Ped dx\\
&-\int_{\R}f(\ued)\Ped dx + \delta \int_{\R}f(\ued)\px\Ped dx\\
&+\eps\int_{\R}\px\ued\Ped dx - \eps\delta\int_{\R}\px\ued\px\Ped dx.
\end{split}
\end{equation}
Due to \eqref{eq:Pmedianulla} and \eqref{eq:F1},
\begin{equation}
\label{eq:F-in-infty}
\begin{split}
2\gamma\int_{\R}\Fed\Ped dx&=2\gamma\int_{\R}\Fed\px\Fed dx =\gamma(\Fed(t,\infty))^2\\
&=\gamma\left( \int_{\R} \Ped(t,x)dx \right)^2=0.
\end{split}
\end{equation}
\eqref{eq:12312} and \eqref{eq:F-in-infty} give
\begin{equation}
\label{eq:12234}
\begin{split}
&\frac{d}{dt}\left(\int_{\R}\Ped^2dx + \delta^2\int_{\R}(\px\Ped)^2dx\right)\\
&\quad=-2\gamma\delta\int_{\R} \Fed\px\Ped dx -2\int_{\R}f(\ued)\Ped dx\\
&\qquad + 2\delta \int_{\R}f(\ued)\px\Ped dx +2\eps\int_{\R}\px\ued\Ped dx\\
&\qquad - 2\eps\delta\int_{\R}\px\ued\px\Ped dx.
\end{split}
\end{equation}
Thanks to \eqref{eq:P-pxP-intfy1}, \eqref{eq:Pmedianulla} and \eqref{eq:F1},
\begin{equation}
\label{eq:346}
-2\delta\gamma\int_{\R}\px\Ped\Fed dx=2\delta\gamma\int_{\R}\Ped\px\Fed dx = 2\delta\gamma\int_{\R} \Ped^2 dx\le 2\gamma \int_{\R} \Ped^2 dx,
\end{equation}
while for \eqref{eq:P-pxP-intfy1},
\begin{equation}
\label{eq:347}
\begin{split}
2\eps\int_{\R}\px\ued\Ped dx=&-2\eps\int_{\R}\ued\px\Ped dx.
\end{split}
\end{equation}
Hence, for \eqref{eq:assflux1}, \eqref{eq:346} and \eqref{eq:347},  we get
\begin{equation*}
\begin{split}
&\frac{d}{dt}\left(\int_{\R}\Ped^2dx + \delta^2\int_{\R}(\px\Ped)^2dx\right)\\
&\quad\le 2\gamma \int_{\R} \Ped^2 dx -2\int_{\R}f(\ued)\Ped dx + 2\delta \int_{\R}f(\ued)\px\Ped dx \\
&\qquad  -2 \eps \int_{\R}\ued\px\Ped dx- 2\eps\delta\int_{\R}\px\ued\px\Ped dx\\
&\quad\le 2\gamma \int_{\R} \Ped^2 dx+ 2\left\vert \int_{\R}f(\ued)\Ped dx\right\vert + 2\delta\left\vert\int_{\R}f(\ued)\px\Ped dx\right\vert\\
&\qquad   +2 \eps \left\vert\int_{\R}\ued\px\Ped dx\right\vert+ 2\eps\delta\left\vert\int_{\R}\px\ued\px\Ped dx\right\vert\\
&\quad\le 2\gamma \int_{\R} \Ped^2 dx+ 2\int_{\R}\vert f(\ued)\vert\vert\Ped \vert dx + 2\delta\int_{\R}\vert f(\ued)\vert \vert\px\Ped\vert dx\\
&\qquad  + 2\eps\int_{\R}\vert\ued\vert\vert\px\Ped\vert dx + 2\eps\delta\int_{\R}\vert\px\ued\vert\vert\px\Ped\vert dx\\
&\quad\le 2\gamma \int_{\R} \Ped^2 dx + 2C_{0}\int_{\R}\vert\Ped \vert\ued^2 dx + 2C_{0}\delta\int_{\R}\vert\px\Ped\vert\ued^2 dx\\
&\qquad + 2\eps\int_{\R}\vert\ued\vert\vert\px\Ped\vert dx + 2\eps\delta\int_{\R}\vert\px\ued\vert\vert\px\Ped\vert dx.
\end{split}
\end{equation*}
For the Young inequality,
\begin{align*}
2\eps\int_{\R}\vert\px\Ped\vert\vert\ued\vert &\le\eps\norm{\px\Ped(t,\cdot)}^2_{L^{2}(\R)}+ \eps \norm{\ued(t,\cdot)}^2_{L^{2}(\R)},\\
2\eps\delta\int_{\R}\vert\px\ued\vert\vert\px\Ped\vert dx&=\int_{\R}\left\vert\frac{\eps\px\ued}{\sqrt{\gamma}}\right\vert\vert 2\sqrt{\gamma}\delta\px\Ped\vert dx\\
& \le \frac{\eps^2}{2\gamma}\norm{\px\ued(t,\cdot)}^2_{L^{2}(\R)} + 2\delta^2\gamma\norm{\px\Ped(t,\cdot)}^2_{L^{2}(\R)}.
\end{align*}
Thus,
\begin{equation}
\label{eq:350}
\begin{split}
\frac{d}{dt}G(t)-2\gamma G(t)\le & \eps\norm{\ued(t,\cdot)}^2_{L^{2}(\R)}+ 2C_{0}\int_{\R}\vert\Ped \vert\ued^2 dx\\  &+2C_{0}\delta\int_{\R}\vert\px\Ped\vert\ued^2 dx + \eps\norm{\px\Ped(t,\cdot)}^2_{L^{2}(\R)}\\
&+ \frac{\eps^{2}}{2\gamma}\norm{\px\ued(t,\cdot)}^2_{L^{2}(\R)},
\end{split}
\end{equation}
where
\begin{equation}
\label{eq:def-di-G}
G(t)=\norm{\Ped(t,\cdot)}^2_{L^2(\R)} + \delta^2\norm{\px\Ped(t,\cdot)}^2_{L^2(\R)}.
\end{equation}
We observe that, for \eqref{eq:l2-u},
\begin{equation}
\label{eq:399}
2C_{0}\int_{\R} \vert\Ped\vert \ued^2 dx\le C_{0}e^{2\gamma t} \norm{\Ped}_{L^{\infty}(I_{T,1})},
\end{equation}
where $I_{T,1}$ is defined in \eqref{eq:defI}.\\
Since $0<\delta <1$, it follows from \eqref{eq:l2-u} and \eqref{eq:h2-P} that
\begin{equation}
\label{eq:400}
\begin{split}
2C_{0}\delta\int_{\R}\vert\px\Ped\vert \ued^2dx&\le 2C_{0}\delta\norm{\px\Ped(t,\cdot)}_{L^{\infty}(\R)}\norm{\ued(t,\cdot)}^2_{L^2(\R)}\\
&\le 2\sqrt{\delta}C_{0}e^{3\gamma t}\le C_{0}e^{3\gamma t}.
\end{split}
\end{equation}
Again by \eqref{eq:h2-P}, we have that
\begin{equation}
\label{eq:401}
\eps\norm{\px\Ped(t,\cdot)}^2_{L^{2}(\R)}\le \eps\norm{\px\Ped(t,\cdot)}^2_{L^{2}(\R)}\le C_{0}e^{2\gamma t}.
\end{equation}
Therefore, \eqref{eq:l2-u}, \eqref{eq:400} and \eqref{eq:401} give
\begin{equation*}
\frac{d}{dt}G(t)-2\gamma G(t)\le  C_{0}\left( \norm{\Ped}_{L^{\infty}(I_{T,1})}+ 1 \right)e^{2\gamma t} + C_{0}e^{3\gamma t}+ \frac{\eps^2}{2\gamma}\norm{\px\ued(t,\cdot)}^2_{L^{2}(\R)}.
\end{equation*}
The Gronwall Lemma, \eqref{eq:u0eps},  \eqref{eq:l2-u} and \eqref{eq:def-di-G} give
\begin{align*}
&\norm{\Ped(t,\cdot)}^2_{L^2(\R)} + \delta^2\norm{\px\Ped(t,\cdot)}^2_{L^2(\R)}\\
&\quad\le \norm{P_{\eps,0}}^2_{L^2(0,\infty)}e^{2\gamma t} +  \left( \norm{\Ped}_{L^{\infty}(I_{T,1})}+ 1 \right)t e^{2\gamma t} + C_{0}t e^{3\gamma t}\\
&\qquad  +\frac{\eps^2 e^{2\gamma t}}{2\gamma}\int_{0}^{t} e^{-2\gamma s} \norm{\px\ued(s,\cdot)}^2_{L^{2}(\R)}ds\\
&\quad \le  \norm{P_{\eps,0}}^2_{L^2(0,\infty)}e^{2\gamma t} +  \left( \norm{\Ped}_{L^{\infty}(I_{T,1})}+ 1 \right)t e^{2\gamma t} + C_{0}t e^{3\gamma t}+C_{0}e^{2\gamma t}.
\end{align*}
Hence,
\begin{equation}
\label{eq:505}
\norm{\Ped(t,\cdot)}^2_{L^2(\R)} + \delta^2\norm{\px\Ped(t,\cdot)}^2_{L^2(\R)}\le C(T)\left( \norm{\Ped}_{L^{\infty}(I_{T,1})}+ 1 \right).
\end{equation}
Due to \eqref{eq:h2-P}, \eqref{eq:505} and the H\"older inequality,
\begin{align*}
\Ped^2(t,x)&\le 2\int_{\R}\vert\Ped\vert\vert\px\Ped\vert dx \le 2\norm{\Ped(t,\cdot)}_{L^2(\R)} \norm{\px\Ped(t,\cdot)}_{L^2(\R)}\\
&\le 2\sqrt{C(T)\left(\norm{\Ped}_{L^{\infty}(I_{T,1})}+1\right)}\sqrt{C_{0}}e^{\gamma t}\le C(T)\left(\norm{\Ped}_{L^{\infty}(I_{T,1})}+1\right).
\end{align*}
Therefore,
\begin{equation*}
\norm{\Ped}^2_{L^{\infty}(I_{T,1})} - C(T) \norm{\Ped}_{L^{\infty}(I_{T,1})} - C(T)\le 0,
\end{equation*}
which gives \eqref{eq:P-infty}.

\eqref{eq:l2P} and \eqref{eq:l2pxP} follow from \eqref{eq:P-infty} and \eqref{eq:505}.

Let us show that \eqref{eq:pt-px-P} holds true. Multiplying \eqref{eq:1554} by $\Ped$, an integration on $\R$ and \eqref{eq:F-in-infty} give
\begin{align*}
2\delta\int_{\R}\ptx\Ped\Ped dx =& \frac{d}{dt} \norm{\Ped(t,\cdot)}^2_{L^2(\R)} -2\gamma\int_{\R}\Fed\Ped dx\\
&+2\int_{\R} f(\ued)\Ped dx -2\eps\int_{\R}\px\ued\Ped dx\\
=&\frac{d}{dt} \norm{\Ped(t,\cdot)}^2_{L^2(\R)}+2\int_{\R} f(\ued)\Ped dx -2\eps\int_{\R}\px\ued\Ped dx.
\end{align*}
An integration on $(0,t)$ gives
\begin{align*}
2\delta\int_{0}^{t}\!\!\!\int_{\R}\ptx\Ped\Ped dx =&\norm{\Ped(t,\cdot)}^2_{L^2(\R)}-\norm{P_{\eps,\delta,0}}^2_{L^2(\R)}\\
&+2\int_{0}^{t}\!\!\!\int_{\R}f(\ued)\Ped dx -2\eps \int_{0}^{t}\!\!\!\int_{\R} \px\ued\Ped dx.
\end{align*}
It follows from \eqref{eq:assflux1}, \eqref{eq:l2-u}, \eqref{eq:P-infty} and \eqref{eq:l2P} that
\begin{align*}
2\delta\left\vert\int_{0}^{t}\!\!\!\int_{\R}\ptx\Ped\Ped dsdx\right\vert \le&\norm{\Ped(t,\cdot)}^2_{L^2(\R)}+\norm{P_{\eps,\delta,0}}^2_{L^2(\R)}\\
&+2\int_{0}^{t}\!\!\!\int_{\R}\vert f(\ued)\vert\vert\Ped\vert dsdx\\
& +2\eps \int_{0}^{t}\!\!\!\int_{\R} \vert\px\ued\vert\vert\Ped\vert dsdx\\
\le & \norm{P_{\eps,\delta,0}}^2_{L^2(\R)}+ 2C(T)\int_{0}^{t}\!\!\!\int_{\R}\ued^2 dsdx\\
 &+2\eps \int_{0}^{t}\!\!\!\int_{\R} \vert\px\ued\vert\vert\Ped\vert dsdx +C(T)\\
\le & \norm{P_{\eps,\delta,0}}^2_{L^2(\R)}  +C(T)\\
&+2\eps \int_{0}^{t}\!\!\!\int_{\R} \vert\px\ued\vert\vert\Ped\vert dsdx.
\end{align*}
Observe that, thanks to \eqref{eq:l2-u},
\begin{equation}
\label{eq:pxU2}
\begin{split}
&\eps\int_{0}^{t}\norm{\px\ued(s,\cdot)}^2_{L^2(\R)}ds\\
&\quad \le \eps e^{2\gamma t}\int_{0}^{t}e^{-2\gamma s}\norm{\px\ued(s,\cdot)}^2_{L^2(\R)}ds\le C(T).
\end{split}
\end{equation}
Due to Young inequality,
\begin{equation}
\label{eq:young}
\begin{split}
&2\eps\int_{\R} \vert\px\ued\vert\vert\Ped\vert dsdx\\
&\quad\le \eps\norm{\Ped(t,\cdot)}^2_{L^2(\R)}+\eps\norm{\px\ued(t,\cdot)}^2_{L^2(\R)}\\
& \quad \le C(T) +\eps\norm{\px\ued(t,\cdot)}^2_{L^2(\R)}.
\end{split}
\end{equation}
Then, for \eqref{eq:pxU2} and  \eqref{eq:young}, we have that
\begin{align*}
&2\eps\int_{0}^{t}\!\!\!\int_{\R}\vert\Ped\vert\vert\px\ued\vert dsdx\\
&\quad \le \eps\int_{0}^{t}\norm{\Ped(s,\cdot)}^2_{L^2(\R)} ds +  \eps\int_{0}^{t}\norm{\px\ued(s,\cdot)}^2_{L^2(\R)} ds\le C(T)  .
\end{align*}
Therefore,
\begin{equation*}
2\delta\left\vert\int_{0}^{t}\!\!\!\int_{\R}\Ped\ptx\Ped dsdx\right\vert\le \norm{P_{\eps,0}}^2_{L^2(\R)}+C(T),
\end{equation*}
which gives \eqref{eq:pt-px-P}.
\end{proof}
\begin{lemma}
\label{lm:linfty-u}
Let $T>0$. Then,
\begin{equation}
\label{eq:linfty-u}
\norm{\ued}_{L^\infty(I_{T,1})}\le \norm{u_{\eps,0}}_{L^\infty(\R)}+ C(T),
\end{equation}
where $I_{T,1}$ is defined in \eqref{eq:defI}.
\end{lemma}
\begin{proof}
Due to \eqref{eq:OHepsw1} and \eqref{eq:P-infty},
\begin{equation*}
\pt \ued +\px f(\ued)-\eps\pxx \ued\le \gamma C(T).
\end{equation*}
Since the map
\begin{equation*}
{\mathcal F}(t):=\norm{u_{\eps,0}}_{L^\infty(\R)}+\gamma C(T)t,
\end{equation*}
solves the equation
\begin{equation*}
\frac{d{\mathcal F}}{dt}=\gamma C(T)
\end{equation*}
and
\begin{equation*}
\max\{\ued(0,x),0\}\le {\mathcal F}(t),\qquad (t,x)\in I_{T,1},
\end{equation*}
the comparison principle for parabolic equations implies that
\begin{equation*}
 \ued(t,x)\le {\mathcal F}(t),\qquad (t,x)\in I_{T,1}.
\end{equation*}
In a similar way we can prove that
\begin{equation*}
\ued(t,x)\ge -{\mathcal F}(t),\qquad (t,x)\in I_{T,1}.
\end{equation*}
Therefore,
\begin{equation*}
\vert\ued(t,x)\vert\le\norm{u_{\eps,0}}_{L^\infty(\R)}+\gamma C(T)t\le\norm{u_{\eps,0}}_{L^\infty(\R)}+ C(T),
\end{equation*}
which gives \eqref{eq:linfty-u}.
\end{proof}

\begin{lemma}\label{lm:34}
Let $T>0$ and $0<\delta<1$. We have that
\begin{equation}
\label{eq:012}
\eps\norm{\px\ued(t,\cdot)}^2_{L^2(\R)} + \eps^2\int_{0}^{t}\norm{\pxx\ued(s,\cdot)}^2_{L^2(\R)}ds\le C(T).
\end{equation}
\end{lemma}
\begin{proof}
Let $0<t<T$. Multiplying \eqref{eq:OHepsw1} by $-\eps\pxx\ued$, we have
\begin{equation}
\label{eq:12346}
\begin{split}
-\eps\pxx\ued\pt\ued &+\eps^2\pxx\ued^2\\
 =&-\gamma\eps\Ped\pxx\ued- \eps f'(\ued)\px\ued\pxx\ued.
\end{split}
\end{equation}
Since
\begin{equation*}
-\eps\int_{\R}\pxx\ued\pt\ued dx=\frac{d}{dt}\left(\frac{\eps}{2}\int_{\R}(\px\ued)^2 \right),
\end{equation*}
integrating \eqref{eq:12346} on $\R$, we get
\begin{align*}
\frac{d}{dt}\left(\eps\int_{\R}(\px\ued)^2 dx\right)&+2\eps^2\int_{\R}(\pxx\ued)^2 dx\\
=& -2\gamma\eps\int_{\R}\Ped\pxx\ued dx \\
&- 2\eps\int_{\R}f'(\ued)\px\ued\pxx\ued dx.
\end{align*}
Due to \eqref{eq:l2-u}, \eqref{eq:l2P}, \eqref{eq:linfty-u} and the Young inequality,
\begin{align*}
&-2\gamma\eps\int_{\R}\Ped\pxx\ued dx\\
&\qquad \le 2\gamma\eps\left\vert\int_{\R}\Ped\pxx\ued dx\right\vert\\
&\qquad \le 2\int_{\R}\left\vert\sqrt{2}\gamma\Ped\right\vert\left\vert\frac{\eps\pxx\ued}{\sqrt{2}}\right\vert dx\\
&\qquad \le 2\gamma^2\norm{\Ped(t,\cdot)}^2_{L^2(\R)} + \frac{\eps^2}{2} \norm{\pxx\ued(t,\cdot)}^2_{L^2(\R)}\\
&\qquad \le C(T)+ \frac{\eps^2}{2}  \norm{\pxx\ued(t,\cdot)}^2_{L^2(\R)},\\
&-2\eps\int_{\R}f'(\ued)\px\ued\pxx\ued dx\\
&\qquad \le 2\eps\left\vert\int_{\R}f'(\ued)\px\ued\pxx\ued dx\right\vert\\
&\qquad \le 2\int_{\R}\left\vert\sqrt{2}f'(\ued)\px\ued\right\vert \left\vert\frac{\eps  \pxx\ued}{\sqrt{2}}\right \vert dx\\
&\qquad \le 2\int_{\R}(f'(\ued))^2(\px\ued^2)+ \frac{\eps^2}{2}\int_{\R}(\pxx\ued)^2 dx\\
&\qquad \le 2\norm{f'}^2_{L^{\infty}(I_{T,2})}\norm{ \px\ued(t,\cdot)}^2_{L^2(\R)}  +  \frac{\eps^2}{2}\norm{\pxx\ued(t,\cdot)}^2_{L^2(\R)},
\end{align*}
where
\begin{equation}
\label{eq:def-I2}
I_{T,2}=\left(-\norm{u_{\eps,0}}_{L^\infty(\R)}-C(T), \norm{u_{\eps,0}}_{L^\infty(\R)}+C(T)\right).
\end{equation}
Therefore,
\begin{align*}
\frac{d}{dt}\left(\eps\norm{\px\ued(t,\cdot)}^2_{L^2(\R)} \right)&+2\eps^2\norm{(\pxx\ued(t,\cdot))}^2_{L^2(\R)}\\
\le & \eps^2  \norm{\pxx\ued(t,\cdot)}^2_{L^2(\R)}+ \norm{f'}^2_{L^{\infty}(I_{T,2})}\norm{ \px\ued(t,\cdot)}^2_{L^2(\R)} + C(T),  \\
\end{align*}
that is
\begin{align*}
\frac{d}{dt}\left(\eps\norm{\px\ued(t,\cdot)}^2_{L^2(\R)} \right)&+\eps^2 \norm{\pxx\ued(t,\cdot)}^2_{L^2(\R)}\\
\le &  \norm{f'}^2_{L^{\infty}(I_{T,2})}\norm{ \px\ued(t,\cdot)}^2_{L^2(\R)}+ C(T).
\end{align*}
An integration on $(0,t)$ and \eqref{eq:u0eps} give
\begin{equation}
\label{eq:001}
\begin{split}
\eps\norm{\px\ued(t,\cdot)}^2_{L^2(\R)} &+ \eps^2\int_{0}^{t}\norm{\pxx\ued(s,\cdot)}^2_{L^2(\R)}ds\\
\le & 2\norm{f'}^2_{L^{\infty}(I_{T,2})}\int_{0}^{t}\norm{ \px\ued(s,\cdot)}^2_{L^2(\R)}ds +C(T).
\end{split}
\end{equation}

\eqref{eq:012} follows from \eqref{eq:pxU2} and \eqref{eq:001}.
\end{proof}
\begin{lemma}\label{lm:37}
Let $T>0$ and $0<\delta<1$. We have that
\begin{equation}
\label{eq:052}
\norm{\px\ued}_{L^{\infty}(I_{T,1})} \le C(T),
\end{equation}
where $I_{T,1}$ is defined in \eqref{eq:defI}. Moreover,
\begin{equation}
\label{eq:055}
\eps\norm{\pxx\ued(t,\cdot)}^2_{L^2(\R)}+\eps^2\int_{0}^{t}\norm{\pxxx\ued(s,\cdot)}^2_{L^2(\R)}ds\le C(T).
\end{equation}

\end{lemma}
\begin{proof}
Let $0<t<T$. Multiplying \eqref{eq:OHepsw1} by $\eps\pxxxx\ued$, we have
\begin{equation}
\label{eq:018}
\begin{split}
\eps\pxxxx\ued\pt\ued &- \eps^2\pxxxx\ued\pxx\ued\\
=&+\eps\gamma\Ped\pxxxx\ued-\eps f'(\ued)\px\ued\pxxxx\ued.
\end{split}
\end{equation}
Since
\begin{align*}
\eps\int_{\R}\pxxxx\ued\pt\ued dx =&\frac{d}{dt}\left(\frac{\eps}{2}\int_{\R}(\pxx\ued)^2 dx\right), \\
- \eps^2\int_{\R}\pxxxx\ued\pxx\ued dx =& \eps^2\int_{\R}(\pxxx\ued)^2 dx,\\
\eps\gamma\int_{\R}\Ped\pxxxx\ued dx = &-\eps\gamma\int_{\R}\px\Ped\pxxx\ued dx,\\
-\eps\int_{\R}f'(\ued)\px\ued\pxxxx\ued dx=&\eps\int_{\R}f''(\ued)(\px\ued)^2\pxxx\ued dx\\
& + \eps\int_{\R}f'(\ued)\pxx\ued\pxxx\ued dx,
\end{align*}
integrating \eqref{eq:12346} on $\R$, we get
\begin{align*}
\frac{d}{dt}\left(\eps\int_{\R}(\pxx\ued)^2 dx\right)&+2\eps^2\int_{\R}(\pxxx\ued)^2 dx\\
=& -2\eps\gamma\int_{\R}\px\Ped\pxxx\ued dx\\
&+2\eps\int_{\R}f''(\ued)(\px\ued)^2\pxxx\ued dx\\
&+ 2\eps\int_{\R}f'(\ued)\pxx\ued\pxxx\ued dx.
\end{align*}
Due to \eqref{eq:h2-P}, \eqref{eq:linfty-u}, \eqref{eq:012} and the Young inequality,
\begin{align*}
&-2\eps\gamma\int_{\R}\px\Ped\pxxx\ued dx\\
&\qquad \le 2 \eps\gamma\left\vert \int_{\R}\px\Ped\pxxx\ued dx \right\vert\\
&\qquad \le 2\int_{\R}\left\vert\sqrt{3}\gamma\px\Ped\right\vert\left\vert\frac{\eps\pxxx\ued}{\sqrt{3}}\right\vert dx\\
&\qquad \le 3\gamma^2\norm{\px\Ped(t.\cdot)}^2_{L^{2}(\R)}+ \frac{\eps^2}{3}\norm{\pxxx\ued(t.\cdot)}^2_{L^{2}(\R)}\\
&\qquad \le C(T) + \frac{\eps^2}{3}\norm{\pxxx\ued(t.\cdot)}^2_{L^{2}(\R)},\\
&2\eps\int_{\R}f''(\ued)(\px\ued)^2\pxxx\ued dx\\
&\qquad \le 2\eps\left\vert\int_{\R}f''(\ued)(\px\ued)^2\pxxx\ued dx\right\vert\\
&\qquad \le 2\int_{\R}\left\vert\sqrt{3} f''(\ued)(\px\ued)^2\right\vert\left\vert\frac{\eps\pxxx\ued}{\sqrt{3}}\right \vert dx\\
&\qquad \le 3\int_{\R}(f''(\ued))^2(\px\ued)^4 dx + \frac{\eps^2}{3}\norm{\pxxx\ued(t,\cdot)}^2_{L^2(\R)}\\
&\qquad \le 3\norm{f''}^2_{L^{\infty}(I_{T,2})}\norm{\px\ued}^2_{L^{\infty}(I_{T,1})}\norm{\px\ued(t,\cdot)}^2_{L^2(\R)}+ \frac{\eps^2}{3} \norm{\pxxx\ued(t,\cdot)}^2_{L^2(\R)}\\
&\qquad \le 3\norm{f''}^2_{L^{\infty}(I_{T,2})}C(T)\norm{\px\ued}^2_{L^{\infty}(I_{T,1})}\\
&\qquad\quad+ \frac{\eps^2}{3} \norm{\pxxx\ued(t,\cdot)}^2_{L^2(\R)},\\
& 2\eps\int_{\R}f'(\ued)\pxx\ued\pxxx\ued dx\\
&\qquad \le 2\eps\left\vert\int_{\R}f'(\ued)\pxx\ued\pxxx\ued dx\right\vert\\
&\qquad \le 2\int_{\R} \left\vert\sqrt{3}f'(\ued)\pxx\ued\right\vert\left \vert \frac{\eps\pxxx\ued}{\sqrt{3}}\right \vert dx\\
&\qquad \le 3\int_{\R}(f'(\ued))^2(\pxx\ued)^2 dx + \frac{\eps^2}{3} \norm{\pxxx\ued(t,\cdot)}^2_{L^2(\R)}\\
&\qquad \le 3\norm{f'}^2_{L^{\infty}(I_{T,2})}\norm{\pxx\ued(t,\cdot)}^2_{L^2(\R)}+ \frac{\eps^2}{3}\norm{\pxxx\ued(t,\cdot)}^2_{L^2(\R)},
\end{align*}
where $I_{T,1}$ is defined in \eqref{eq:defI} and $I_{T,2}$ is defined in \eqref{eq:def-I2}. Therefore,
\begin{align*}
\frac{d}{dt}\left(\eps\norm{\pxx\ued(t,\cdot)}^2_{L^2(\R)}\right)&+2\eps^2\norm{\pxxx\ued(t,\cdot)}^2_{L^2(\R)}\\
\le &\eps^2\norm{\pxxx\ued(t.\cdot)}^2_{L^{2}(\R)}\\
&+3\norm{f''}^2_{L^{\infty}(I_{T,2})}C(T)\norm{\px\ued}^2_{L^{\infty}(I_{T,1})}\\
&+ 3\norm{f'}^2_{L^{\infty}(I_{T,2})}\norm{\pxx\ued(t,\cdot)}^2_{L^2(\R)} +C(T),
\end{align*}
that is
\begin{align*}
\frac{d}{dt}\left(\eps\norm{\pxx\ued(t,\cdot)}^2_{L^2(\R)}\right)&+\eps^2\norm{\pxxx\ued(t,\cdot)}^2_{L^2(\R)}\\
\le & C(T)\norm{\px\ued}^2_{L^{\infty}(I_{T,1})} +C(T)\\
& + C(T)\norm{\pxx\ued(t,\cdot)}^2_{L^2(\R)}.
\end{align*}
An integration on $(0,t)$, \eqref{eq:u0eps} and \eqref{eq:012} give
\begin{align*}
\eps\norm{\pxx\ued(t,\cdot)}^2_{L^2(\R)}&+\eps^2\int_{0}^{t}\norm{\pxxx\ued(s,\cdot)}^2_{L^2(\R)}ds\\
\le &\left(C(T)\norm{\px\ued}^2_{L^{\infty}(I_{T,1})} +C(T)\right)\int_{0}^{t}ds\\
& + C(T)\int_{0}^{t}\norm{\pxx\ued(s,\cdot)}^2_{L^2(\R)}ds\\
\le & C(T)\norm{\px\ued}^2_{L^{\infty}(I_{T,1})} +C(T).
\end{align*}
Thus,
\begin{equation}
\label{eq:045}
\begin{split}
\eps\norm{\pxx\ued(t,\cdot)}^2_{L^2(\R)}&+\eps^2\int_{0}^{t}\norm{\pxxx\ued(s,\cdot)}^2_{L^2(\R)}ds\\
 \le& C(T)\left(1+\norm{\px\ued}^2_{L^{\infty}(I_{T,1})}\right).
\end{split}
\end{equation}
Due to \eqref{eq:012}, \eqref{eq:045} and the H\"older inequality,
\begin{align*}
(\px\ued(t,x))^2\le& 2\int_{\R}\vert\px\ued\vert\vert\pxx\ued\vert dx\\
\le &  2\norm{\px\ued(t,\cdot)}_{L^2(\R)} \norm{\pxx\ued(t,\cdot)}_{L^2(\R)}\\
\le & C(T)\sqrt{\left(1+\norm{\px\ued}^2_{L^{\infty}(I_{T,1})}\right)}.
\end{align*}
Then,
\begin{equation*}
\norm{\px\ued}^4_{L^{\infty}(I_{T,1})} -C(T)\norm{\px\ued}^2_{L^{\infty}(I_{T,1})} -C(T) \le 0,
\end{equation*}
which gives \eqref{eq:052}. \\
\eqref{eq:055} follows from \eqref{eq:052} and \eqref{eq:045}.
\end{proof}
Arguing as in \cite{CHK:ParEll}, we obtain the following result
\begin{lemma}\label{lm:38}
Let $T>0$, $\ell >2$  and $0<\delta <1$. For each $t\in (0,T)$,
\begin{equation}
\px^{\ell}\ued(t,\cdot)\in L^2(\R).
\end{equation}
\end{lemma}
We are in a position to state and prove the following result.
\begin{lemma}\label{lm:exist}
Let $T>0$. Assume \eqref{eq:assflux1}, \eqref{eq:assinit1}, \eqref{eq:L-2Peps0} and \eqref{eq:def-di-Peps0}. Then there exist
\begin{align}
\label{eq:uePe1}
\ue&\in L^{\infty}((0,T)\times\R)\cap C((0,T);H^\ell(\R)),\quad \ell>2,\\
\label{eq:uePe2}
\Pe&\in L^{\infty}((0,T)\times\R)\cap L^{2}((0,T)\times\R),
\end{align}
where $\ue$ is a classic solution of the Cauchy problem of \eqref{eq:OHepsw}.
\end{lemma}

\begin{proof}
Let $\eta:\R\to\R$ be any convex $C^2$ entropy function, and
$q:\R\to\R$ be the corresponding entropy
flux defined by $q'=f'\eta'$.
By multiplying the first equation in \eqref{eq:OHepsw1} with
$\eta'(\ue)$ and using the chain rule, we get
\begin{equation*}
    \pt  \eta(\ued)+\px q(\ued)
    =\underbrace{\eps \pxx \eta(\ued)}_{=:\CL_{1,\delta}}
    \, \underbrace{-\eps \eta''(\ued)\left(\px  \ued\right)^2}_{=: \CL_{2,\delta}}
     \, \underbrace{+\gamma\eta'(\ued) \Ped}_{=: \CL_{3,\delta}},
\end{equation*}
where  $\CL_{1,\delta}$, $\CL_{2,\delta}$, $\CL_{3,\delta}$ are distributions.\\

Let us show that
\begin{equation}
\label{eq:compct-H-1}
\textrm{$\{\CL_{1,\delta}\}_\delta$ is compact in $H^{-1}((0,T)\times\R)$, $T>0$}.
\end{equation}
Since
\begin{equation*}
\eps\pxx\eta(\ued)=\px(\eps\eta'(\ued)\px\ued),
\end{equation*}
we have to prove that
\begin{align}
\label{eq:eta12}
&\textrm{$\{\eps\eta'(\ued)\px\ued\}_\delta$ is bounded in $L^2((0,T)\times\R)$, $T>0$},\\
\label{eq:eta13}
&\textrm{$\{\eps\eta''(\ued)(\px\ued)^2+ \eps\eta'(\ued)\pxx\ued\}_\delta$ is bounded in $L^2((0,T)\times\R)$, $T>0$}.
\end{align}
We begin by proving that \eqref{eq:eta12} holds true.
Thanks to Lemmas \ref{lm:l2-u} and \ref{lm:linfty-u},
\begin{align*}
\norm{\eps\eta'(\ued)\px\ued}^2_{L^2((0,T)\times\R)}&\le\eps ^2\norm{\eta'}^2_{L^{\infty}(I_{T,2})}\int_{0}^{T}\norm{\px\ued(s,\cdot)}^2_{L^2(\R)}ds\\
&\le \eps^2\norm{\eta'}^2_{L^{\infty}(I_{T,2})} e^{2\gamma T}\int_{0}^{T}e^{-2\gamma s}\norm{\px\ued(s,\cdot)}^2_{L^2(\R)}ds\\
&\le \frac{\eps}{2}\norm{\eta'}^2_{L^{\infty}(I_{T,2})}e^{2\gamma T}\norm{u_{\eps,0}}^2_{L^2(\R)}\le C(T),
\end{align*}
where $I_{T,2}$ is defined in \eqref{eq:def-I2}.

We claim that
\begin{equation}
\label{eq:eta14}
\textrm{$\{\eps\eta''(\ued)(\px\ued)^2\}_\delta$ is bounded in $L^2((0,T)\times\R)$}.
\end{equation}
Due to Lemmas \ref{lm:l2-u}, \ref{lm:linfty-u}, \ref{lm:37}
\begin{align*}
\norm{\eps\eta''(\ued)(\px\ued)^2}^2_{L^2((0,T)\times\R)}&\le\eps ^2\norm{\eta''}^2_{L^{\infty}(I_{T,2})}\int_{0}^{T}\!\!\!\!\int_{\R}(\px\ued(s,x))^4dsdx\\
&\le \eps^2\norm{\eta''}^2_{L^{\infty}(I_{T,2})}\norm{\px\ued}^2_{L^{\infty}(I_{T,1})}\int_{0}^{T}\norm{\px\ued(s,\cdot)}^2_{L^2(\R)}ds\\
&\le \frac{\eps}{2}\norm{\eta''}^2_{L^{\infty}(I_{T,2})}\norm{\px\ued}^2_{L^{\infty}(I_{T,1})}e^{2\gamma T}\norm{u_{\eps,0}}^2_{L^2(\R)}\le C(T),
\end{align*}
where $I_{T,1}$ is defined in \eqref{eq:defI}.

We claim that
\begin{equation}
\label{eq:eta15}
\textrm{$\{\eps\eta'(\ued)\pxx\ued\}_\delta$ is bounded in $L^2((0,T)\times\R)$}.
\end{equation}
Thanks to Lemmas \ref{lm:linfty-u} and \ref{lm:34},
\begin{align*}
\norm{\eps\eta'(\ued)\pxx\ued}^2_{L^2((0,T)\times\R)}&\le\eps ^2\norm{\eta'}^2_{L^{\infty}(I_{T,2})}\int_{0}^{T}\norm{\pxx\ued(s,\cdot)}^2_{L^2(\R)}ds\\
&\le \eps^2\norm{\eta'}^2_{L^{\infty}(I_{T,2})}C(T)\le C(T).
\end{align*}
\eqref{eq:eta14} and \eqref{eq:eta15} give \eqref{eq:eta13}.\\
Therefore, \eqref{eq:compct-H-1} follows from \eqref{eq:eta12} and \eqref{eq:eta13}.

We have that
\begin{equation*}
\textrm{$\{\CL_{2,\delta}\}_{\delta>0}$ is bounded in $L^1((0,T)\times\R)$}.
\end{equation*}
Due to Lemmas \ref{lm:l2-u}, \ref{lm:linfty-u},
\begin{align*}
\norm{\eps\eta''(\ued)(\px\ued)^2}_{L^1((0,T)\times\R)}&\le\eps\norm{\eta''}_{L^{\infty}(I_{T,2})}\int_{0}^{T}\norm{\px\ued(s,\cdot)}^2_{L^2(\R)}ds\\
&\le \eps\norm{\eta'}^2_{L^{\infty}(I_{T,2})} e^{2\gamma T}\int_{0}^{T}e^{-2\gamma s}\norm{\px\ued(s,\cdot)}^2_{L^2(\R)}ds\\
&\le \frac{\norm{\eta'}^2_{L^{\infty}(I_{T,2})}e^{2\gamma T}}{2}\norm{u_{\eps,0}}^2_{L^2(\R)}\le C(T).
\end{align*}
We have that
\begin{equation*}
\textrm{$\{\CL_{3,\delta}\}_{\delta>0}$ is bounded in $L^1_{loc}((0,T)\times\R)$.}
\end{equation*}
Let $K$ be a compact subset of $(0,T)\times\R$. For Lemmas \ref{lm:P-infty} and \ref{lm:linfty-u},
\begin{align*}
\norm{\gamma\eta'(\ued)\Ped}_{L^1(K)}&=\gamma\int_{K}\vert\eta'(\ue)\vert\vert\Pe\vert
dtdx\\
&\leq \gamma
\norm{\eta'}_{L^{\infty}(I_{T,2})}\norm{\Pe}_{L^{\infty}(I_{T,1})}\vert K \vert .
\end{align*}
Therefore, Murat's lemma \cite{Murat:Hneg} implies that
\begin{equation}
\label{eq:GMC1}
    \text{$\left\{  \pt  \eta(\ued)+\px q(\ued)\right\}_{\delta>0}$
    lies in a compact subset of $\Hneg((0,\infty)\times\R)$.}
\end{equation}
The $L^{\infty}$ bound stated in Lemma \ref{lm:linfty-u}, \eqref{eq:GMC1} and the
 Tartar's compensated compactness method \cite{TartarI} give the existence of a subsequence
$\{\uedk\}_{k\in\N}$ and a limit function $\ue\in L^{\infty}((0,T)\times\R)$
such that
\begin{equation}\label{eq:convu}
    \textrm{$\uedk \to \ue$ a.e.~and in $L^{p}_{loc}((0,T)\times\R)$, $1\le p<\infty$}.
\end{equation}
Hence,
\begin{equation}
\label{eq:udelta-to-ueps}
 \textrm{$\uedk \to \ue$  in $L^{\infty}((0,T)\times\R)$}.
\end{equation}
Moreover, for convexity, we have
\begin{equation}
\label{eq:H-2-R}
\begin{split}
\norm{\ue(t,\cdot)}_{L^2(\R)}^2+2\eps e^{2\gamma t}\int_0^t e^{-2\gamma s}\norm{\px \ue(s,\cdot)}^2_{L^2(\R)}ds&\le C(T),\\
\eps\norm{\px\ue(t,\cdot)}^2_{L^2(\R)} + \eps^2\int_{0}^{t}\norm{\pxx\ue(s,\cdot)}^2_{L^2(\R)}ds&\le C(T),\\
\eps\norm{\pxx\ue(t,\cdot)}^2_{L^2(\R)}+\eps^2\int_{0}^{t}\norm{\pxxx\ue(s,\cdot)}^2_{L^2(\R)}ds&\le C(T).
\end{split}
\end{equation}
We need only to observe that
\begin{align*}
&2\eps e^{2\gamma t}\int_0^t e^{-2\gamma s}\norm{\px \ue(s,\cdot)}^2_{L^2(\R)}ds\\
&\quad\le 2\eps e^{2\gamma t}\liminf_{k}\int_0^t e^{-2\gamma s}\norm{\px \uedk(s,\cdot)}^2_{L^2(\R)}ds\le C(T),\\
&\eps^2\int_{0}^{t}\norm{\pxx\ue(s,\cdot)}^2_{L^2(\R)}ds\le \eps^2 \liminf_{k} \int_{0}^{t}\norm{\pxx\uedk(s,\cdot)}^2_{L^2(\R)}ds\le C(T),\\
&\eps^2\int_{0}^{t}\norm{\pxxx\ue(s,\cdot)}^2_{L^2(\R)}ds\le \eps^2\liminf_{k}\int_{0}^{t}\norm{\pxxx\ued(s,\cdot)}^2_{L^2(\R)}ds\le C(T).
\end{align*}
Moreover, it follows from convexity and Lemma \ref{lm:38} that
\begin{equation}
\label{eq:1420}
\px^{\ell}\ue(t,\cdot)\in L^{2}(\R),\quad \ell >2, \quad t\in (0,T).
\end{equation}
Therefore, \eqref{eq:udelta-to-ueps}, \eqref{eq:H-2-R} and \eqref{eq:1420} give \eqref{eq:uePe1}. \eqref{eq:uePe2} follows from Lemma \eqref{lm:P-infty}.

Finally, we prove that
\begin{equation}
\label{eq:pxp-u}
\int_{-\infty}^{x}\ue(t,y)dy=\Pe(t,x), \quad \textrm{a.e.} \quad \textrm{in} \quad (t,x)\in I_{T,1}.
\end{equation}
Integrating the second equation of \eqref{eq:OHepsw1} on $(-\infty,x)$, for \eqref{eq:P-pxP-intfy1}, we have that
\begin{equation}
\label{eq:00113}
\int_{-\infty}^{x}\uedk(t,y)dy= \Pedk(t,x) -\dk\px\Pedk(t,x).
\end{equation}
We show that
\begin{equation}
\label{eq:px-1}
\textrm{$\delta\px\Ped(t,x)\to 0$ in $L^{\infty}(0,T; L^{\infty}(0,\infty))$, $T>0$ as $\delta\to0$.}
\end{equation}
It follows from \eqref{eq:h2-P} that
\begin{equation*}
\delta\norm{\px\Ped}_{L^{\infty}(0,T; L^{\infty}(0,\infty))}\leq \sqrt{\delta}e^{\gamma t}\norm{u_{\eps,0}}_{L^2(\R)} =\sqrt{\delta}C(T)\to 0,
\end{equation*}
that is \eqref{eq:px-1}.\\
Therefore, \eqref{eq:pxp-u} follows from \eqref{eq:uePe1}, \eqref{eq:uePe2}, \eqref{eq:00113} and \eqref{eq:px-1}.
The proof is done.
\end{proof}
\begin{lemma}\label{lm:u-null}
Let $\ue(t,x)$ be a classic solution of \eqref{eq:OHepsw}. Then,
\begin{equation}
\label{eq:con-u}
\int_{\R} \ue(t,x) dx=0, \quad t\ge 0,
\end{equation}
\end{lemma}
\begin{proof}
Differentiating \eqref{eq:OHepsw} with respect to $x$, we have
\begin{equation}
\label{eq:diff-oh}
\px(\pt\ue + \px f(\ue) - \eps \pxx \ue)=\gamma \ue.
\end{equation}
Since $\ue$ is a smooth solution of \eqref{eq:OHepsw}, an integration over $\R$ gives \eqref{eq:con-u}.
\end{proof}
We are ready for the proof of Theorem \ref{th:wellp}.
\begin{proof}[Proof of Theorem \ref{th:wellp}]
Lemma \ref{lm:exist} gives the existence of a classic solution $\ue(t,x)$ of \eqref{eq:OHepsw}, or \eqref{eq:OHepswint}.

Let us show that $\ue(t,x)$ is unique and \eqref{eq:l2-stability} holds true.
Let $\ue,\ve$ be two classic solution of \eqref{eq:OHepsw}, or \eqref{eq:OHepswint}, that is
\begin{align*}
& \begin{cases}
\pt\ue+ f'(\ue)\px\ue=\gamma\Pe^{\ue}+\eps\pxx\ue,& t>0,  x\in\R,\\
\px\Pe^{\ue}=\ue, & t>0, x\in\R,\\
\ue(0,x)=u_{\eps, 0}(x),& x\R,
\end{cases}\\
&\begin{cases}
\pt\ve+f'(\ve)\px\ve=\gamma\Pe^{\ve} +\eps\pxx\ve, & t>0, x\in\R,\\
\px\Pe^{\ve}=\ve, & t>0, x\in\R,\\
\ve(0,x)=v_{\eps,0}(x),& x\in\R.
\end{cases}
\end{align*}
Then, the function
\begin{equation}
\label{eq:def-di-omega}
\oeps(t,x)=\ue(t,x)-\ve(t,x)
\end{equation}
is solution of the following Cauchy problem
\begin{equation}
\label{eq:epsw}
\begin{cases}
\pt \oeps+ f'(\ue)\px\ue -f'(\ve)\px\ve =\gamma\Oeps+ \eps\pxx\oeps,&\quad t>0,\ x\in\R,\\
\px\Oeps=\oeps,&\quad t>0,\ x\in\R,\\
\oeps(0,x)=u_{\eps,0}(x) - v_{\eps,0}(x) ,&\quad x\in\R,
\end{cases}
\end{equation}
where
\begin{equation}
\label{eq:def-di-Omega}
\begin{split}
\Oeps(t,x)&=\Pe^{\ue}(t,x)- \Pe^{\ve}(t,x)\\
&=\int_{-\infty}^{x} \ue(t,y)dy - \int_{-\infty}^{x} \ve(t,y)dy\\
&=\int_{-\infty}^{x} (\ue(t,y)-\ve(t,y))dy =\int_{-\infty}^x\oeps(t,y)dy.
\end{split}
\end{equation}
It follows from Lemma \ref{lm:u-null} and \eqref{eq:def-di-Omega} that
\begin{equation}
\label{eq:Omega-in-infty}
\Oeps(t,\infty)= \int_{\R} \ue(t,y)dy - \int_{\R} \ve(t,y)dy=0.
\end{equation}
Observe that, for \eqref{eq:def-di-omega},
\begin{align*}
f'(\ue)\px\ue-f'(\ve)\px\ve &= f'(\ue)\px\ue - f'(\ue)\px\ve + f'(\ue)\px\ve - f'(\ve)\px\ve\\
&=f'(\ue)\px(\ue -\ve) +\px\ve ( f'(\ue)- f'(\ve))\px\ve\\
&=f'(\ue)\px\oeps + ( f'(\ue)- f'(\ve))\px\ve.
\end{align*}
Therefore, the first equation of \eqref{eq:epsw} is equivalent to the following one:
\begin{equation}
\label{eq:epsw1}
\pt \oeps+ f'(\ue)\px\oeps + ( f'(\ue)- f'(\ve))\px\ve =\gamma\Oeps+ \eps\pxx\oeps.
\end{equation}
Moreover, since $\ue$ and $\ve$ are in $L^{\infty}((0,T)\times\R)$, we have
that
\begin{equation}
\label{eq:f1}
\Big\vert f'(\ue(t,x))- f'(\ve(t,x))\Big\vert \leq C(T) \vert \ue(t,x) - \ue(t,x)\vert,\quad (t,x)\in (0,T)\times\R,
\end{equation}
where
\begin{equation}
\label{eq:sup}
C(T)=\sup_{(0,T)\times\R}\Big\{\vert f''(\ue)\vert + \vert f''(\ve)\vert\Big\}.
\end{equation}
Therefore, \eqref{eq:def-di-omega} and \eqref{eq:f1} give
\begin{equation}
\label{eq:f1omega}
\Big\vert f'(\ue(t,x))- f'(\ve(t,x))\Big\vert \leq C(T) \vert\oeps(t,x)\vert,\quad (t,x)\in (0,T)\times\R
\end{equation}
Multiplying \eqref{eq:epsw1} by $\oeps$, and integration on $\R$ gives
\begin{align*}
\frac{d}{dt}\int_{\R} \oeps^2dx=&2\int_{\R} \oeps\pt\oeps dx\\
=&2\eps\int_{\R}\oeps\pxx\oeps dx-2\int_{\R}\oeps f'(\ue)\px\oeps  dx\\
&-2\int_{\R}\oeps( f'(\ue)- f'(\ve))\px\ve dx  +2\gamma\int_{\R}\Oeps\oeps dx\\
=& -2\eps\int_{\R}(\px\oeps)^2 dx +\int_{\R}\oeps^2 f''(\ue)\px\ue dx\\
&-2\int_{\R}\oeps( f'(\ue)- f'(\ve))\px\ve dx  +2\gamma\int_{\R}\Oeps\oeps  dx.
\end{align*}
It follows from the second equation of \eqref{eq:epsw} and Lemma \ref{lm:u-null} that
\begin{equation}
\label{eq:1230}
\begin{split}
&\frac{d}{dt}\norm{\oeps(t,\cdot)}^2_{L^{2}(\R)}+2\eps\norm{\px\oeps(t,\cdot)}^2_{L^{2}(\R)}\\
&\quad\le \int_{\R}\oeps^2 \vert f''(\ue)\vert\vert\px\ue\vert dx + 2\int_{\R}\vert\oeps\vert\vert( f'(\ue)- f'(\ve))\vert\vert\px\ve\vert dx.
\end{split}
\end{equation}
Since $\ue(t,\cdot),\ve(t,\cdot)\in H^{\ell}(\R),\ell >2$, for each $t\in (0,T)$, then
\begin{equation}
\label{eq:1231}
\px\ue(t,\cdot),\px\ve(t,\cdot)\in H^{\ell-1}(\R)\subset L^{\infty}(\R), \quad t\in (0,T).
\end{equation}
Therefore, thanks to \eqref{eq:f1}, \eqref{eq:sup}, \eqref{eq:1230} and \eqref{eq:1231},
\begin{align*}
\frac{d}{dt}\norm{\oeps(t,\cdot)}^2_{L^{2}(\R)}+2\eps\norm{\px\oeps(t,\cdot)}^2_{L^{2}(\R)}\le C(T)\norm{\oeps(t,\cdot)}^2_{L^{2}(\R)}.
\end{align*}
The Gronwall Lemma gives
\begin{equation}
\label{eq:gro1}
\norm{\oeps(t,\cdot)}^2_{L^{2}(\R)}+2\eps e^{C(T)t}\int_{0}^{s}e^{-C(T)s} \norm{\px\oeps(s,\cdot)}^2_{L^{2}(\R)} ds \le e^{C(T)t}\norm{\omega_{\eps,0}}^2_{L^2(\R)}.
\end{equation}
Hence, \eqref{eq:l2-stability} follows from \eqref{eq:def-di-omega}, \eqref{eq:epsw} and \eqref{eq:gro1}.
\end{proof}

\section{Existence of  entropy solutions for Ostrovshy-Hunter Equation}\label{sec:Es}
This section is devoted to the existence of entropy solutions for \eqref{eq:OHw-u}, or \eqref{eq:OHw}.

Fix a small number $\eps>0$, and let $\ue=\ue (t,x)$ be the unique classical solution of \eqref{eq:OHepsw}, where  $u_{\eps,0}$ is a $C^\infty(\R)$ approximation of $u_{0}$ such that
\begin{equation}
\begin{split}
\label{eq:u0eps1}
&\norm{u_{\eps,0}}_{L^2(\R)}\le \norm{u_0}_{L^2(\R)},\quad \norm{u_{\eps,0}}_{L^\infty(\R)}\le \norm{u_0}_{L^\infty(\R)},\\
&\int_{\R} u_{\eps,0}(x) dx =0, \quad\int_{\R}\left(\int_{-\infty}^{x}u_{\eps,0}(y) dy\right)^2 dx  \le \norm{P_0}^2_{L^2(\R)},\\
&\int_{\R}\left(\int_{-\infty}^{x}u_{\eps,0}(y) dy\right)dx=\int_{\R}P_{\eps,0}(x)dx =0,
\end{split}
\end{equation}
where
\begin{equation}
\label{eq:def-di-peps0}
P_{\eps,0}(x)=\int_{-\infty}^{x}u_{\eps,0}(y) dy,
\end{equation}
and $\displaystyle \norm{P_0}_{L^2(\R)}$ is defined in \eqref{eq:L-2P0}.

Let us prove some a priori estimates on $\ue$ and $\Pe$, denoting with $C(T)$ the constants which depend on $T$, but independent on $\eps$.

Following \cite[Lemma $6$]{CdK}, or \cite[Lemma $2.3.1$]{dR}, we show this result.
\begin{lemma}\label{lm:cns2}
Let us suppose that, for each $t\ge 0$,
\begin{equation}
\label{eq:P-int-1}
\textrm{$\Pe(t,x)$ is integrable at $-\infty$, (or  at $+\infty$)},
\end{equation}
where $\Pe(t,x)$ is defined in \eqref{eq:OHepsw}. Then, the following statements are equivalent:
\begin{align}
\label{eq:con-u-1}
\int_{\R} \ue(t,x) dx =&0,\\
\label{eq:stima-l-2-ueps}
\norm{\ue(t,\cdot)}^2_{L^2(\R)} + 2\eps\int_{0}^{t}\norm{\px\ue(s,\cdot)}^2_{L^2(\R)} ds=&\norm{u_{\eps,0}}^2_{L^2(\R)},\\
\label{eq:Pmedianulla1}
\int_{\R} \Pe(t,x) dx =&0,\\
\label{eq:stima-l-2-pesp}
\norm{\Pe(t,\cdot)}^2_{L^2(\R)}+2\eps\norm{\ue(t,\cdot)}^2_{L^2(\R)}=& \norm{P_{\eps,0}}^2_{L^2(\R)}-2\int_{0}^{t}\!\!\int_{\R}\Pe f(\ue)dsdx,
\end{align}
for every $t\ge 0$.
\end{lemma}
\begin{proof}
Let $t> 0$. We begin by proving that \eqref{eq:con-u-1} implies \eqref{eq:stima-l-2-ueps}. \\
Multiplying  \eqref{eq:OHepswint} by $\ue$, an integration on $\R$ gives
\begin{align*}
\frac{1}{2}\frac{d}{dt}\int_{\R} \ue^2dx=&\int_{\R} \ue\pt\ue dx\\
=&\eps\int_{\R}\ue\pxx\ue dx-\int_{\R} \ue f'(\ue)\px\ue dx+\gamma\int_{\R}\ue\Big(\int_{-\infty}^{x}\ue dy\Big)dx\\
= & - \eps\int_{\R} (\px\ue)^2 dx + \gamma\int_{\R}\ue\Big(\int_{-\infty}^{x}\ue dy\Big)dx.
\end{align*}
For \eqref{eq:OHepsw},
\begin{equation*}
\int_{\R}\ue\Big(\int_{-\infty}^{x}\ue
dy\Big)dx=\int_{\R}\Pe(t,x)\px\Pe(t,x)dx=\frac{1}{2}\Pe ^2(t,\infty).
\end{equation*}
Then,
\begin{equation}
\label{eq:121}
\frac{d}{dt}\int_{\R} \ue^2dx + 2\eps\int_{\R} (\px\ue)^2 dx=\gamma\Pe ^2(t,\infty).
\end{equation}
Thanks to \eqref{eq:con-u-1},
\begin{equation}
\label{eq:122}
\lim_{x\to\infty}\Pe ^2(t,x)= \left(\int_{\R}\ue(t,x)dx\right)^2=0.
\end{equation}
\eqref{eq:121}, \eqref{eq:122} and an integration on $(0,t)$  give \eqref{eq:stima-l-2-ueps}.

Let us show that \eqref{eq:stima-l-2-ueps} implies \eqref{eq:con-u-1}.
We assume by contradiction that \eqref{eq:con-u-1} does not hold, namely:
\begin{equation}
\label{eq:unot0}
\int_{\R} \ue(t,x) dx \neq 0.
\end{equation}
For \eqref{eq:OHepsw},
\begin{equation*}
\Pe^2(t,\infty)=\Big(\int_{\R} \ue(t,x) dx\Big)^2\neq 0.
\end{equation*}
Therefore, \eqref{eq:121} and an integration on $(0,t)$ give
\begin{equation*}
\norm{\ue(t,\cdot)}^2_{L^2(\R)}  + 2\eps\int_{0}^{t}\norm{\px\ue(s,\cdot)}^2_{L^2(\R)} ds\neq\norm{u_{\eps,0}}^2_{L^2(\R)},
\end{equation*}
which is in contradiction with \eqref{eq:stima-l-2-ueps}.

Let us show that \eqref{eq:con-u-1} implies \eqref{eq:Pmedianulla1}.
We begin by observing that, for \eqref{eq:P-int-1},  we can consider the following function:
\begin{equation}
\label{eq:F11}
\Fe(t,x)=\int_{-\infty}^x\Pe(t,y) dy.
\end{equation}
Thanks to the regularity of $\ue$ and \eqref{eq:F11}, integrating on $(-\infty,x)$ the first equation of \eqref{eq:OHepsw}, we get
\begin{equation*}
\int_{-\infty}^x\pt\ue(t,y) dy+f(\ue(t,x)) - f(0)-\eps\px\ue(t,x)=\gamma\Fe(t,x),
\end{equation*}
that is
\begin{equation}
\label{eq:1300}
\frac{d}{dt}\int_{-\infty}^x\ue(t,y) dy+f(\ue(t,x))-f(0)-\eps\px\ue(t,x)=\gamma\Fe(t,x).
\end{equation}
Instead, from the second equation of \eqref{eq:OHepsw}, we have
\begin{equation}
\label{eq:1301}
\pt\Pe(t,x)=\frac{d}{dt}\int_{-\infty}^{x}\ue(t,y)dy.
\end{equation}
It follows from \eqref{eq:1300} and \eqref{eq:1301} that
\begin{equation}
\label{eq:equaP}
\pt\Pe(t,x)+f(\ue(t,x))-f(0)+\eps\px\ue(t,x)=\gamma\Fe(t,x).
\end{equation}
We observe that, for \eqref{eq:con-u-1} and \eqref{eq:1301},
\begin{equation}
\label{eq:Pe10}
\lim_{x\to\infty} \pt\Pe(t,x) =\int_{\R} \pt\ue(t,x)dx=\frac{d}{dt}\int_{\R}\ue(t,x)dx=0,
\end{equation}
while for the regularity of $\ue$,
\begin{equation}
\label{eq:500-1}
\lim_{x\to\infty}\Big( f(\ue(t,x))-f(0)-\eps\px\ue(t,x)\Big)=0.
\end{equation}
Therefore, for \eqref{eq:F11}, \eqref{eq:equaP}, \eqref{eq:Pe10} and \eqref{eq:500-1},  we get
\begin{equation}
\label{eq:F-in-infty-1}
\Fe(t,\infty)=\int_{\R} \Pe(t,x)dx=0,
\end{equation}
that is \eqref{eq:Pmedianulla1}.

Let us show that \eqref{eq:Pmedianulla1} implies \eqref{eq:con-u-1}.
We assume by contradiction that \eqref{eq:con-u-1} does not hold, that is \eqref{eq:unot0}.
Then, for \eqref{eq:1301},
\begin{equation}
\label{eq:1305}
\lim_{x\to\infty} \pt\Pe(t,x) =\int_{\R} \pt\ue(t,x)dx=\frac{d}{dt}\int_{\R}\ue(t,x)dx\neq 0.
\end{equation}
It follows from \eqref{eq:F11}, \eqref{eq:equaP}, \eqref{eq:500-1} and \eqref{eq:1305} that
\begin{equation*}
\int_{\R} \Pe(t,x)dx \neq 0,
\end{equation*}
which is in contradiction with \eqref{eq:Pmedianulla1}.

Let us show that \eqref{eq:Pmedianulla1} implies \eqref{eq:stima-l-2-pesp}.
Multiplying  \eqref{eq:equaP} by $\Pe$, an integration on $\R$ gives
\begin{align*}
\frac{1}{2}\frac{d}{dt}\int_{\R} \Pe^2dx=&\int_{\R} \Pe\pt\Pe dx\\
=&\eps\int_{\R}\px\ue\Pe dx-\int_{\R} \Pe f(\ue) dx + f(0)\int_{\R} \Pe dx +\gamma\int_{\R}\Pe\Fe dx,
\end{align*}
that is
\begin{equation*}
\frac{d}{dt}\int_{\R} \Pe^2dx=2\eps\int_{\R}\px\ue\Pe dx-2\int_{\R} \Pe f(\ue) dx + f(0)\int_{\R} \Pe dx + 2\gamma\int_{\R}\Pe\Fe dx.
\end{equation*}
For \eqref{eq:OHepsw},
\begin{equation*}
2\eps\int_{\R}\px\ue\Pe dx= -2\eps\int_{\R}\ue\px\Pe dx=-2\eps\int_{\R}\ue^2dx,
\end{equation*}
while for \eqref{eq:F11},
\begin{equation*}
2\int_{\R}\Pe\Fe dx=2\int_{\R}\Fe\px\Fe dx= \Fe^2(t,\infty)-\Fe^2(t,-\infty)=\Fe^2(t,\infty).
\end{equation*}
Then,
\begin{equation}
\label{eq:300}
\frac{d}{dt}\int_{\R} \Pe^2dx=-2\eps\int_{\R}\ue^2dx -2\int_{\R} \Pe f(\ue) dx +f(0)\int_{\R}\Pe dx  +\gamma\Fe ^2(t,\infty).
\end{equation}
Thanks to \eqref{eq:Pmedianulla1},
\begin{equation}
\begin{split}
\label{eq:301}
\lim_{x\to\infty}\Fe ^2(t,x)&= \Big(\int_{\R}\Pe(t,x)dx\Big)^2=0,\\
f(0)\int_{\R}\Pe dx&=0.
\end{split}
\end{equation}
\eqref{eq:300}, \eqref{eq:301} and an integration on $(0,t)$ give \eqref{eq:stima-l-2-pesp}.

Let us show that \eqref{eq:stima-l-2-pesp} implies \eqref{eq:Pmedianulla1}.
We assume by contradiction that \eqref{eq:Pmedianulla1} does not hold, namely:
\begin{equation*}
\int_{\R} \Pe(t,x) dx \neq 0.
\end{equation*}
For \eqref{eq:F11},
\begin{equation}
\label{eq:24}
\Fe^2(t,\infty)=\Big(\int_{\R} \Pe(t,x) dx\Big)^2\neq 0.
\end{equation}
Moreover,
\begin{equation}
\label{eq:23}
f(0)\int_{\R}\Pe dx\neq 0.
\end{equation}
Therefore, \eqref{eq:300}, \eqref{eq:24} and \eqref{eq:23} gives
\begin{equation*}
\frac{d}{dt}\int_{\R} \Pe^2dx +2\eps\int_{\R}\ue^2dx +\int_{\R} \Pe\ue^2 dx \neq 0 ,
\end{equation*}
that is
\begin{equation*}
\frac{d}{dt}\int_{\R} \Pe^2dx \neq -2\eps\int_{\R}\ue^2dx -\int_{\R} \Pe f(\ue) dx .
\end{equation*}
Therefore, we have that
\begin{equation*}
\norm{\Pe(t,\cdot)}^2_{L^2(\R)}+2\eps\int_{0}^{t}\norm{\ue(s,\cdot)}^2_{L^2(\R)} ds\neq \norm{P_{\eps,0}}^2_{L^2(\R)}-2\int_{0}^{t}\!\!\int_{\R}\Pe f(\ue)dsdx,
\end{equation*}
which is a contradiction with \eqref{eq:stima-l-2-pesp}.
\end{proof}

\begin{lemma}
\label{lm:l2-u-1}
For each $t \ge 0$, \eqref{eq:P-int-1} and \eqref{eq:Pmedianulla1} hold true.
Moreover, we have that
\begin{align}
\label{eq:l2-u-1}
\norm{\ue(t,\cdot)}_{L^2(\R)}^2+2\eps \int_0^t \norm{\px \ue(s,\cdot)}^2_{L^2(\R)}ds&\le\norm{u_{0}}^2_{L^2(\R)},\\
\label{eq:l2-P-1}
\norm{\Pe(t,\cdot)}^2_{L^2(\R)}+2\eps\norm{\ue(t,\cdot)}^2_{L^2(\R)}&\le \norm{P_{0}}^2_{L^2(\R)}+2C_{0}\int_{0}^{t}\!\!\int_{\R}\vert\Pe\vert \ue^2dsdx.
\end{align}
\end{lemma}
\begin{proof}
We begin by proving that \eqref{eq:P-int-1} holds true.
Let $a$ be, an arbitrary real number. Integrating on $(a,x)$ the second equation of \eqref{eq:OHepsw}, we get
\begin{equation*}
\int_{a}^{x} \ue(t,y) dy = \Pe(t,x)-\Pe(t,a).
\end{equation*}
Since $\Pe(t,-\infty)=0$, then
\begin{equation}
\label{eq:30}
\int_{a}^{-\infty} \ue(t,x) dx= -\Pe(t,a).
\end{equation}
Differentiating \eqref{eq:30} with respect to $t$, we get
\begin{equation}
\label{eq:31}
\frac{d}{dt}\int_{a}^{-\infty} \ue(t,x) dx=\int_{a}^{-\infty}\pt \ue(t,x) dx = -\pt\Pe(t,a).
\end{equation}
Integrating on $(a,x)$ the first equation of \eqref{eq:OHepsw}, we obtain that
\begin{equation}
\label{eq:34}
\begin{split}
\int_{a}^{x}\pt\ue(t,y)dy &+ f(\ue(t,x)) - f(\ue(t,a))\\
&-\eps\px\ue(t,x)+\eps\px\ue(t,a)=\gamma\int_{a}^{x}\Pe(t,y) dy.
\end{split}
\end{equation}
Being $\ue$ a smooth solution of \eqref{eq:OHepswint}, we have that
\begin{equation}
\label{eq:500-2}
\lim_{x\to-\infty}\Big( f(\ue(t,x))-\eps\px\ue(t,x)\Big)=f(0).
\end{equation}
It follows from \eqref{eq:31}, \eqref{eq:34} and \eqref{eq:500-2} that
\begin{equation*}
\gamma\int_{a}^{-\infty}\Pe(t,x) dx= -\pt\Pe(t,a)+ f(0)- f(\ue(t,a))+\eps\px\ue(t,a),
\end{equation*}
which gives \eqref{eq:P-int-1}.
Therefore, for Lemmas \ref{lm:u-null} and \ref{lm:cns2}, we have \eqref{eq:Pmedianulla1}.
Lemmas \ref{lm:u-null} and \ref{lm:cns2} also say that \eqref{eq:stima-l-2-ueps} holds true. Thus, \eqref{eq:l2-u-1} follows from \eqref{eq:u0eps1} and \eqref{eq:stima-l-2-ueps}.

Finally, we prove \eqref{eq:l2-P-1}. Again by Lemmas \ref{lm:u-null} and \ref{lm:cns2}, we get \eqref{eq:stima-l-2-pesp}. Then, for \eqref{eq:assflux1} and  \eqref{eq:u0eps1},
\begin{align*}
\norm{\Pe(t,\cdot)}^2_{L^2(\R)}+2\eps\norm{\ue(t,\cdot)}^2_{L^2(\R)}&= \norm{P_{\eps,0}}^2_{L^2(\R)}-2\int_{0}^{t}\!\!\int_{\R}\Pe f(\ue)dsdx\\
&\le \norm{P_{0}}^2_{L^2(\R)}+ 2\left\vert \int_{0}^{t}\!\!\int_{\R}\Pe f(\ue)dsdx\right\vert\\
&\le \norm{P_{0}}^2_{L^2(\R)}+ 2\int_{0}^{t}\!\!\int_{\R}\vert\Pe\vert \vert f(\ue)\vert dsdx\\
&\le \norm{P_{0}}^2_{L^2(\R)}+ 2C_{0}\int_{0}^{t}\!\!\int_{\R}\vert\Pe\vert\ue^2 dsdx,
\end{align*}
that is \eqref{eq:l2-P-1}.
\end{proof}
\begin{lemma}\label{lm:p-infty-1}
Let $T>0$. There exists a function $C(T)>0$, independent on $\eps$, such that
\begin{align}
\label{eq:P-infty-1}
\norm{\Pe}_{L^{\infty}(I_{T,1})}&\le C(T),\\
\label{eq:l2P-1}
\norm{\Pe(t,\cdot)}_{L^2(\R)}&\le C(T),
\end{align}
where $I_{T,1}$ is defined in \eqref{eq:defI}.
\end{lemma}
\begin{proof}
We begin by observing that, for \eqref{eq:stima-l-2-ueps} and \eqref{eq:stima-l-2-pesp},
\begin{align*}
\norm{\Pe(t,\cdot)}^2_{L^2(\R)}&\le \norm{P_{0}}^2_{L^2(\R)} + 2C_{0} \norm{u_{0}}^2_{L^2(\R)}\norm{\Pe}_{L^{\infty}(I_{T,1})}t\\
&\le \norm{P_{0}}^2_{L^2(\R)} + 2C_{0} \norm{u_{0}}^2_{L^2(\R)}T\norm{\Pe}_{L^{\infty}(I_{T,1})}.
\end{align*}
Hence,
\begin{equation}
\label{eq:p-l2}
\norm{\Pe(t,\cdot)}^2_{L^2(\R)}\le \norm{P_{0}}^2_{L^2(\R)}+C_{1}(T)\norm{\Pe}_{L^{\infty}(I_{T,1})}.
\end{equation}
Due to the H\"older inequality, we get
\begin{equation*}
\Pe^2(t,x)\leq 2\int_{\R}\vert\Pe\px\Pe\vert dx\leq 2\norm{\Pe(t,\cdot)}_{L^2(\R)}\norm{\px\Pe(t,\cdot)}_{L^2(\R)},
\end{equation*}
that is
\begin{equation*}
\Pe^4(t,x)\leq 4\norm{\Pe(t,\cdot)}^2_{L^2(\R)}\norm{\px\Pe(t,\cdot)}^2_{L^2(\R)}.
\end{equation*}
For \eqref{eq:OHepsw}, \eqref{eq:stima-l-2-ueps} and \eqref{eq:p-l2},
\begin{align*}
\Pe^4(t,x)\leq& 4\norm{\Pe(t,\cdot)}^2_{L^2(\R)}\norm{u_{0}}^2_{L^2(\R)}\\
\leq & 4 \norm{P_{0}}^2_{L^2(\R)}\norm{u_{0}}^2_{L^2(\R)}+4\norm{u_{0}}^2_{L^2(\R)}C_{1}(T)\norm{\Pe}_{L^{\infty}(I_{T,1})}.
\end{align*}
Therefore,
\begin{equation}
\label{eq:L-4}
\norm{\Pe}^4_{L^{\infty}(I_{T,1})}-C_{2}(T)\norm{\Pe}_{L^{\infty}(I_{T,1})} - 4 \norm{P_{0}}^2_{L^2(\R)}\norm{u_{0}}^2_{L^2(\R)}\leq 0.
\end{equation}
Let us consider the following function
\begin{equation}
\label{eq:g-func}
g(X)=X^4 - C_{2}(T)X -A,
\end{equation}
where
\begin{equation}
\label{eq:cost1}
A=4 \norm{P_{0}}^2_{L^2(\R)}\norm{u_{0}}^2_{L^2(\R)}>0.
\end{equation}
We observe that
\begin{equation}
\label{eq:teozeri}
\lim_{X\to -\infty} g(X)=\infty, \quad g(0)=-A<0.
\end{equation}

Since $g'(X)=4 X^3 -C_{2}(T) $, we have that
\begin{align*}
g\quad \textrm{is increasing in}\quad (E(T), \infty),
\end{align*}
where $\displaystyle E(T)=\Big(\frac{C_{2}(T)}{4}\Big)^{\frac{1}{3}}>0$.\\
Thus,
\begin{equation}
\label{eq:600}
g(E(T))<g(0)<0.
\end{equation}
Moreover,
\begin{equation}
\label{eq:7000}
\lim_{X\to \infty} g(X)=\infty.
\end{equation}
Then, it follows from \eqref{eq:teozeri}, \eqref{eq:600} and \eqref{eq:7000} that the function $g$ has only two zeros $D(T)<0<C(T)$.
Therefore, the inequality
\begin{equation*}
X^4 - C_{2}(T)X -A\leq 0
\end{equation*}
is verified when
\begin{equation}
\label{eq:601}
D(T)\leq X\leq C(T).
\end{equation}
Taking $X=\norm{\Pe}_{L^{\infty}(I_{T,1})}$, we have \eqref{eq:P-infty-1}.

Finally, \eqref{eq:l2P-1} follows from \eqref{eq:P-infty-1} and \eqref{eq:p-l2}.
\end{proof}
Arguing as Section \ref{sec:vv}, Lemma \ref{lm:linfty-u},  we obtain the following result
\begin{lemma}
\label{lm:linfty-u-1}
Let $T>0$. Then,
\begin{equation}
\label{eq:linfty-u-1}
\norm{\ue}_{L^\infty(I_{T,1})}\le\norm{u_0}_{L^\infty(\R)}+C(T),
\end{equation}
where $I_{T,1}$ is defined in \eqref{eq:defI}.
\end{lemma}

Let us continue by proving the existence of  a distributional solution
to  \eqref{eq:OH}, \eqref{eq:init}  satisfying \eqref{eq:OHentropy}.
\begin{lemma}\label{lm:conv}
Let $T>0$. There exists a function $u\in L^{\infty}((0,T)\times\R)$ that is a distributional
solution of \eqref{eq:OHw} and satisfies  \eqref{eq:OHentropy} for every convex entropy $\eta\in C^2(\R)$.
\end{lemma}
We  construct a solution by passing
to the limit in a sequence $\Set{u_{\eps}}_{\eps>0}$ of viscosity approximations \eqref{eq:OHepsw}. We use the compensated compactness method \cite{TartarI}.
\begin{lemma}\label{lm:conv-u-1}
Let $T>0$. There exists a subsequence
$\{\uek\}_{k\in\N}$ of $\{\ue\}_{\eps>0}$
and a limit function $  u\in L^{\infty}((0,T)\times\R)$
such that
\begin{equation}\label{eq:convu-1}
    \textrm{$\uek \to u$ a.e.~and in $L^{p}_{loc}((0,T)\times\R)$, $1\le p<\infty$}.
\end{equation}
Moreover, we have
\begin{equation}
\label{eq:conv-P-1}
\textrm{$\Pek \to P$ in $ L^{\infty}((0,T)\times\R)\cap L^{2}((0,T)\times\R)$}
\end{equation}
such that
\begin{equation}
\label{eq:Px-u1}
\textrm{$\px P= u$ in the sense of distributions on $[0,\infty)\times\R$.}
\end{equation}
\end{lemma}
\begin{proof}
Let $\eta:\R\to\R$ be any convex $C^2$ entropy function, and
$q:\R\to\R$ be the corresponding entropy
flux defined by $q'=f'\eta'$.
By multiplying the first equation in \eqref{eq:OHepsw} with
$\eta'(\ue)$ and using the chain rule, we get
\begin{equation*}
    \pt  \eta(\ue)+\px q(\ue)
    =\underbrace{\eps \pxx \eta(\ue)}_{=:\CL_{1,\eps}}
    \, \underbrace{-\eps \eta''(\ue)\left(\px  \ue\right)^2}_{=: \CL_{2,\eps}}
     \, \underbrace{+\gamma\eta'(\ue) \Pe}_{=: \CL_{3,\eps}},
\end{equation*}
where  $\CL_{1,\eps}$, $\CL_{2,\eps}$, $\CL_{3,\eps}$ are distributions.\\
Let us show that
\begin{equation*}
\textrm{$\CL_{1,\eps}\to 0$ in $H^{-1}((0,T)\times\R)$, $T>0$}.
\end{equation*}
Since
\begin{equation*}
\eps\pxx\eta(\ue)=\px(\eps\eta'(\ue)\px\ue),
\end{equation*}
for Lemmas \ref{lm:l2-u-1} and \ref{lm:linfty-u-1},
\begin{align*}
\norm{\eps\eta'(\ue)\px\ue}^2_{L^2((0,T)\times\R)}&\le\eps ^2\norm{\eta'}^2_{L^{\infty}(J_T)}\int_{0}^{T}\norm{\px\ue(s,\cdot)}^2_{L^2(\R)}ds\\
&\le
\frac{\eps}{2}\norm{\eta'}^2_{L^{\infty}(J_T)}\norm{u_{0}}^2_{L^2(\R)}\to
0,
\end{align*}
where
\begin{equation*}
J_T=\left(-\norm{u_0}_{L^\infty(\R)}- C(T), \norm{u_0}_{L^\infty(\R)}+ C(T)\right).
\end{equation*}
Arguing as Lemma \eqref{lm:exist}, we obtain that
\begin{align*}
&\textrm{$\{\CL_{2,\eps}\}_{\eps>0}$ is uniformly bounded in $L^1((0,T)\times\R), \quad T>0$},\\
&\textrm{$\{\CL_{3,\eps}\}_{\eps>0}$ is uniformly bounded in $L^1_{loc}((0,T)\times\R),\quad T>0$}.
\end{align*}
Therefore, Murat's lemma \cite{Murat:Hneg} implies that
\begin{equation}
\label{eq:GMC11}
    \text{$\left\{  \pt  \eta(\ue)+\px q(\ue)\right\}_{\eps>0}$
    lies in a compact subset of $\Hneg((0,\infty)\times\R)$.}
\end{equation}
The $L^{\infty}$ bound stated in Lemma \ref{lm:linfty-u-1}, \eqref{eq:GMC11} and the
 Tartar's compensated compactness method \cite{TartarI} give the existence of a subsequence
$\{\uek\}_{k\in\N}$ and a limit function $  u\in L^{\infty}((0,T)\times\R)$
such that \eqref{eq:convu-1} holds.

\eqref{eq:conv-P-1} follows from Lemma \ref{lm:p-infty-1}.

We conclude by proving that \eqref{eq:Px-u1} holds true.\\
Let $\phi\in C^{\infty}(\R^2)$ be a test function with compact support. Multiplying by $\phi$ the second equation of \eqref{eq:OHepsw}, we have that
\begin{equation*}
\int_{0}^{\infty}\!\!\!\!\int_{\R}\px\Pek\phi dtdx = \int_{0}^{\infty}\!\!\!\!\int_{\R} \uek dsdx,
\end{equation*}
that is
\begin{equation}
\label{eq:40}
-\int_{0}^{\infty}\!\!\!\!\int_{\R}\Pek\px\phi dtdx= \int_{0}^{\infty}\!\!\!\!\int_{\R} \uek \phi dsdx.
\end{equation}
\eqref{eq:Px-u1} follows from  \eqref{eq:convu-1}, \eqref{eq:conv-P-1} and \eqref{eq:40}.
\end{proof}

\section{Oleinik estimate and uniqueness of the entropy solution for Ostrovsky-Hunter equation}\label{sec:olei-uni}
In \cite{CdK, dR}, it is  proved that the initial value problem \eqref{eq:OHw-u}, or \eqref{eq:OHw}, admits a unique entropy solution, when the flux is assumed
Lipschitz continuous. Denoting with $C(T)$ the constants which depends on $T$,  in this section, we prove the following theorem.
\begin{theorem}[Oleinik estimate]\label{th:olei}
Fixed $T>0$. Let us suppose that the flux $f$ is strictly convex, that is
\begin{equation}
\label{eq:fconvex}
f''\ge c>0, \quad \text{for some constant c}.
\end{equation}
Then, there exists a positive constant $C(T)$ such that
\begin{equation}
\label{eq:stima-olei}
\frac{u(t,x)-u(t,y)}{x-y}\le C(T)\left(\frac{1}{t}+1\right),
\end{equation}
for almost every $t\in(0,T)$ and $x,\,y\in\R$, $x\neq y$, where $u$ is the unique entropy weak solution of \eqref{eq:OHw-u}, or \eqref{eq:OHw}.
\end{theorem}
\begin{proof}
Fixed $T>0$, let $\ue$ be the solution of \eqref{eq:OHepsw}, or \eqref{eq:OHepswint}. We claim that there exists a positive constant $C(T)$ such that
\begin{equation}
\label{eq:stima-ux}
\px\ue(t,x)\le C(T)\left(\frac{1}{t}+1\right), \quad t\in (0,T),\, x\in\R.
\end{equation}
Differentiating with respect to $x$ the equation in \eqref{eq:OHepswint}, we obtain
\begin{equation*}
\ptx\ue + f'(\ue)\pxx\ue + f''(\ue)(\px\ue)^2 -\eps\pxxx\ue=\gamma\ue.
\end{equation*}
Let us consider the Cauchy problem
\begin{equation}
\label{eq:PCv}
\begin{cases}
\pt v + f'(\ue)\px v + f''(\ue)v^2 -\eps\pxx v = \gamma\ue, \quad &t\in (0,T),\, x\in\R,\\
v(0,x)=\px u_{0,\eps}(x), \quad &x\in\R.
\end{cases}
\end{equation}
Clearly, the solution of \eqref{eq:PCv} is $\px\ue$.
Due to \eqref{eq:linfty-u-1} and \eqref{eq:fconvex},
\begin{align*}
\pt v + f'(\ue)\px v & -\eps\pxx v=\gamma \ue - f''(\ue)v^2\le C(T) -c v^2.
\end{align*}
Therefore, a supersolution of \eqref{eq:PCv} satisfies the following ordinary differential equation
\begin{equation}
\label{eq:equation-f}
\frac{d z}{d t} +c z^2 - C(T)=0, \quad z(0)=\norm{\px u_{0,\eps}}_{L^{\infty}(\R)}.
\end{equation}
We consider the map
\begin{equation*}
Z(t)= \frac{1}{ct}+ \sqrt{\frac{C(T)}{c}}, \quad t\in (0,T).
\end{equation*}
Observe that
\begin{equation*}
\frac{d Z}{d t} +c Z^2 - C(T)= -\frac{1}{ct^2}+ c\left(\frac{1}{ct}+ \sqrt{\frac{C(T)}{c}}\right)^2 - C(T)= \frac{2  \sqrt{\frac{C(T)}{c}}}{t}\ge 0.
\end{equation*}
Then, for every $t\in (0,T)$, $Z(t)$ is a supersolution of \eqref{eq:equation-f}.
The comparison principle for parabolic equation and the comparison principle for ordinary differential equations give
\begin{equation*}
\px\ue(t,x) \le z(t)\le Z(t)=\frac{1}{ct}+ \sqrt{\frac{C(T)}{c}}, \quad t\in (0,T),\, x\in\R,
\end{equation*}
that is \eqref{eq:stima-ux}.
Since for every $t\in (0, T)$ and $x,\,y\in\R$, $x\neq y$, thanks to \eqref{eq:stima-ux},
\begin{equation*}
\frac{\ue(t,x)-\ue(t,y)}{x-y}=\frac{1}{x-y}\int_{y}^{x} \px\ue(t,\xi)d\xi = \frac{1}{ct} + \sqrt{\frac{C(T)}{c}}\le C(T)\left (\frac{1}{t} +1\right).
\end{equation*}
\eqref{eq:convu-1} gives \eqref{eq:stima-olei}.
\end{proof}
Let us assume that there exist two bounded distributional solution $u$ and $v$ of \eqref{eq:OHw-u}, or \eqref{eq:OHw}, such that
\begin{equation}
\label{eq:olei212}
\frac{u(t,x)-u(t,y)}{x-y}\le C(T)\left(\frac{1}{t} +1\right), \quad \frac{v(t,x)-v(t,y)}{x-y}\le C(T)\left(\frac{1}{t} +1\right),
\end{equation}
for almost every $0< t< T$, $x,\,y \in \R$, $x\ne y$, and some constant $C(T)>0$. We want to prove that
\begin{equation}
\label{eq:u-uguale-v-1}
u=v \quad \textrm{a.e.} \quad \textrm{in} \quad (0,T)\times\R.
\end{equation}
Let $\phi\in C^{\infty}(\R^2)$ be a test function with compact support. Since $u$ and $v$ are distributional solutions of \eqref{eq:OHw-u}, or \eqref{eq:OHw}, we have that
\begin{align*}
&\int_{0}^{\infty}\!\!\!\!\!\int_{\R}(u\pt\phi + f(u)\px\phi)dtdx +\gamma \int_{0}^{\infty}\!\!\!\!\!\int_{\R}\phi\left(\int_{-\infty}^{x} u(t,x) dy\right)dtdx + \int_{\R}\phi(0,x)u_{0}(x)dx=0,\\
&\int_{0}^{\infty}\!\!\!\!\!\int_{\R}(v\pt\phi + f(v)\px\phi)dtdx + \gamma\int_{0}^{\infty}\!\!\!\!\!\int_{\R}\phi\left(\int_{-\infty}^{x} v(t,x) dy\right)dtdx + \int_{\R}\phi(0,x)u_{0}(x)dx=0
\end{align*}
and then
\begin{align*}
&\int_{0}^{\infty}\!\!\!\!\!\int_{\R}\left((u-v)\pt\phi +(f(u)-f(v))\px\phi\right)dtdx\\
&\quad\quad+ \gamma\int_{0}^{\infty}\!\!\!\!\!\int_{\R}\phi\left(\int_{-\infty}^{x} (u(t,y)-v(t,y)) dy\right)dtdx=0.
\end{align*}
Therefore, we have that
\begin{equation}
\label{eq:equat-int-1}
\int_{0}^{\infty}\!\!\!\!\!\int_{\R}w(\pt\phi+ b\px\phi)dsdx+ \gamma\int_{0}^{\infty}\!\!\!\!\!\int_{\R}\phi\left(\int_{-\infty}^{x}w(t,y)dy\right)dtdx=0,
\end{equation}
where
\begin{equation}
\label{eq:def-di-b}
w=u-v, \quad b(t,x)=\int_{0}^{1}f'(\theta u(t,x) + (1-\theta)v(t,x))d\theta=\frac{f(u(t,x))-f(v(t,x))}{u(t,x)-v(t,x)}.
\end{equation}
Observe that, since
\begin{equation}
\label{eq:v1230}
u,v\in L^{\infty}((0,T)\times\R),
\end{equation}
we get
\begin{equation}
\label{eq:f-infty}
\vert f(u(t,x))-f(v(t,x))\vert \le C(T) \vert u(t,x)-v(t,x)\vert,
\end{equation}
where
\begin{equation*}
C(T)=\sup_{(0,T)\times\R}\Big\{\vert f'(u)\vert + \vert f'(v)\vert\Big\}.
\end{equation*}
Therefore, for \eqref{eq:stima-olei} and \eqref{eq:f-infty}, on the function $b(t,x)$, we have the following estimates
\begin{equation}
\label{eq:stime-per-b}
\begin{split}
\norm{b}_{L^{\infty}((0,T)\times\R)}\le& C(T), \quad T>0,\\
\frac{b(t,x)-b(t,y)}{x-y}\le& C(T)\left(\frac{1}{t}+1\right), \quad x\neq y, \quad 0<t<T.
\end{split}
\end{equation}
Now, let us consider the following set:
\begin{equation*}
\Omega :=\{(x,y)\in \R^2;\quad y\le x\}.
\end{equation*}
Therefore,
\begin{align*}
\int_{0}^{\infty}\!\!\!\!\!\int_{\R}\phi\left(\int_{-\infty}^{x}w(t,y)dy\right)dsdx&=\int_{0}^{\infty}\!\!\!\!\!\int_{\Omega}\phi w(t,y)dtdxdy\\
&=\int_{0}^{\infty}\!\!\!\!\!\int_{\R}w(t,y)\left(\int_{y}^{\infty} \phi dx\right)dtdy.
\end{align*}
Hence,
\begin{equation}
\label{eq:int-phi-1}
\int_{0}^{\infty}\!\!\!\!\!\int_{\R}\phi\left(\int_{-\infty}^{x}w(t,y)dy\right)dtdx=\int_{0}^{\infty}\!\!\!\!\!\int_{\R}\Phi(t,y)w(t,y)dtdy,
\end{equation}
where
\begin{equation}
\label{eq:def-di-Phi-1}
\Phi(t,y)= \int_{y}^{\infty}\phi(t,y)dy
\end{equation}
It follows from \eqref{eq:equat-int-1} and \eqref{eq:int-phi-1} that
\begin{equation}
\label{eq:test-fu-3}
\int_{0}^{\infty}\!\!\!\!\!\int_{\R}w(\pt\phi+ b\px\phi + \gamma\Phi)dsdx=0.
\end{equation}
Fix $\displaystyle\psi \in C_{c}((0,\infty)\times\R)$ and let $\tau >0$ be such that
\begin{equation}
\label{eq:con-psi}
\supp(\psi)\subset (0,\tau)\times \R,
\end{equation}
to have \eqref{eq:u-uguale-v-1}, we have to solve the following system:
\begin{equation}
\label{eq:test-funct-1}
\begin{cases}
\pt\phi + b\px\phi= \psi-\gamma\Phi, \quad &(t,x)\in (0,\tau)\times\R\\
\px\Phi=-\phi,\quad &(t,x)\in (0,\tau)\times\R\\
\phi(\tau,x)=0,\quad & x\in\R.
\end{cases}
\end{equation}
We coin \eqref{eq:test-funct-1} the adjoint problem associated with \eqref{eq:OHepsw}.

The idea is to solve \eqref{eq:test-funct-1} and then pass from \eqref{eq:test-fu-3} to the following equation
\begin{equation}
\label{eq:test-func-4}
\int_{0}^{\infty}\!\!\!\!\!\int_{\R}w\psi dtdx=0.
\end{equation}
Unfortunately, due to the low regularity of the coefficient $b$, we cannot solve directly \eqref{eq:test-funct-1}. Hence, we regularize the first equation by smoothing the coefficient $b$ by convolution and adding an artificial viscosity term.

The use of an adjoint problem to prove uniqueness is rather common in the context of first order conservation laws, see for example \cite{CK, LX, Olei, Sm, Ta}.

Let us consider $\{\rho_{\eps}(t,x)\}_{\eps>0}$ a sequence of standard mollifiers. Define
\begin{equation*}
\be=b*\rho_{\eps}, \quad \eps>0,
\end{equation*}
where $*$ denotes the convolution in both variables $t$ and $x$.

Clearly, from \eqref{eq:assflux1}, \eqref{eq:def-di-b} and \eqref{eq:stime-per-b},
\begin{align}
\label{eq:b1}
\be\to b, &\quad \textrm{in}\quad  L^2((0,T)\times\R),\quad  T>0,\\
\label{eq:b2}
\norm{\be}_{L^{\infty}((0,T)\times\R)}\le C(T), &\quad T,\eps>0,\\
\label{eq:b3}
\px\be(t,x)\le C(T)\left(\frac{1}{t}+1\right), &\quad 0<t<T, \quad x\in\R, \quad \eps>0.
\end{align}
Now, we approximate \eqref{eq:test-funct-1} in following way:
\begin{equation}
\label{eq:test-funct-10}
\begin{cases}
\pt\phi_{\eps} + \be\px\phi_{\eps}= \psi_{\eps}-\gamma\Phi_{\eps}-\eps\pxx\phi_{\eps}, \quad &(t,x)\in (0,\tau)\times\R\\
\px\Phi_{\eps}=-\phi_{\eps},\quad &(t,x)\in (0,\tau)\times\R\\
\phi_{\eps}(\tau,x)=0,\quad & x\in\R.
\end{cases}
\end{equation}
The existence of solutions for \eqref{eq:test-funct-10} is obtained considering the following system
\begin{equation*}
\begin{cases}
\pt\phi_{\eps,\delta} + b_{\eps,\delta}\px\phi_{\eps,\delta}= \psi_{\eps,\delta}-\gamma\Phi_{\eps,\delta}-\eps\pxx\phi_{\eps,\delta}, \quad &(t,x)\in (0,\tau)\times\R\\
-\delta\pxx\Phi_{\eps,\delta}+\px\Phi_{\eps,\delta}=-\phi_{\eps,\delta},\quad &(t,x)\in (0,\tau)\times\R\\
\phi_{\eps,\delta}(\tau,x)=0,\quad & x\in\R,
\end{cases}
\end{equation*}
and sending $\delta\to 0$ (see Section \ref{sec:vv}).

Therefore, arguing as in Section \ref{sec:vv}, Theorem \ref{th:wellp}, we obtain
\begin{lemma}
\label{lm:dual-proble}
Let $\eps>0$ and suppose $\psi_{\eps}\in C^{\infty}((0,\infty)\times\R)\cap C((0,\infty);H^2(\R))$ obeys \eqref{eq:con-psi}. There exists a unique solution $\phi_{\eps}\in C^{\infty}((0,\infty)\times\R)\cap C((0,\infty);H^{\ell}(\R))$,  $\ell>2$, to the terminal value problem \eqref{eq:test-funct-10}.
\end{lemma}
Since we feel more comfortable with initial value problems, we define
\begin{align}
\label{eq:v-q}
\we(t,x)= \phi_{\eps}(\tau-t, x), &\quad Q_{\eps}(t,x)=\Phi_{\eps}(\tau-t, x)\\
\label{eq:beta-psi}
\betas(t,x)=\be(\tau-t,x), &\quad \psit(\tau-t, x)=-\psi_{\eps}(\tau-t,x),
\end{align}
for $(t,x)\in (0,\tau)\times\R$. Due to Lemma \ref{lm:dual-proble}, $\we$ is then the unique smooth solution of the initial value problem
\begin{equation}
\label{eq:equ-v}
\begin{cases}
\pt \we - \betas\px \we= \psit+\gamma Q_{\eps}+\eps\pxx \we, \quad &(t,x)\in (0,\tau)\times\R\\
\px Q_{\eps}=-\we,\quad &(t,x)\in (0,\tau)\times\R\\
\we(0,x)=0,\quad & x\in\R.
\end{cases}
\end{equation}
Denoting with $C(\tau)$ the constants which depends on $\tau$, thanks to \eqref{eq:b2}, \eqref{eq:b3} and \eqref{eq:beta-psi}, we get
\begin{align}
\label{eq:beta1}
\norm{\betas}_{L^{\infty}((0,\tau)\times\R)}&\le C(\tau), \quad \eps>0,\\
\label{eq:beta2}
\px\betas(t,x)&\le C(\tau)\left(\frac{1}{\tau -t}+1\right), \quad (t,x)\in (0,\tau)\times\R, \quad \eps >0.
\end{align}
We prove our key estimates.
\begin{lemma}
Let $\psi_{\eps}\in C^{\infty}((0,\infty)\times\R)\cap C((0,\infty);H^{2}(\R))\cap L^{\infty}((0,\infty);H^{2}(\R))$ be a function satisfying \eqref{eq:con-psi}.
Then, using the notation introduced in \eqref{eq:v-q} and \eqref{eq:beta-psi}, there exists a function $C(\tau)>0$, independent on $\eps$ such that
\begin{equation}
\label{eq:v-213}
\begin{split}
\norm{\we(t,\cdot)}^2_{H^1(\R)}&+2\eps \int_{0}^{t}\norm{\px\we(s,\cdot)}^2_{H^1(\R)}ds\\
\le& e^{C(\tau)\tau}\left(\frac{\tau}{\tau-t}\right)^{C(\tau)}\int_{0}^{t}\norm{\psit(s,\cdot)}^2_{H^1(\R)}ds,
\end{split}
\end{equation}
for every $t\in(0,\tau)$.
In particular, we have that
\begin{equation}
\label{eq:v-infty}
\norm{\we(t,\cdot)}_{L^{\infty}(\R)}\le\sqrt{2}\left( e^{C(\tau)\tau}\left(\frac{\tau}{\tau-t}\right)^{C(\tau)}\int_{0}^{t}\norm{\psit(s,\cdot)}^2_{H^1(\R)}ds\right)^{\frac{1}{2}}.
\end{equation}
\end{lemma}
\begin{proof}
We begin by proving that \eqref{eq:v-213} holds true.
Multiplying the first equation of \eqref{eq:equ-v} by $v$, an integration on $\R$ gives
\begin{equation}
\label{eq:v-l2}
\frac{d}{dt}\int_{\R}\we^2 dx=2\int_{\R}\betas \we\px\we dx + 2\int_{\R}\we\psit dx+ 2\gamma\int_{\R} Q_{\eps} \we dx + 2\eps \int_{\R}\we\pxx \we dx.
\end{equation}
Thanks the second equation of \eqref{eq:equ-v}, we have that
\begin{equation}
\label{eq:v-media-nulla}
2\gamma\int_{\R} Q_{\eps}(t,x) \we(t,x) dx = -2\gamma\int_{\R}Q_{\eps}(t,x)\px Q_{\eps}(t,x) dx =0.
\end{equation}
Therefore, it follows from \eqref{eq:beta1}, \eqref{eq:v-l2}, \eqref{eq:v-media-nulla} and the Young inequality that
\begin{equation}
\label{eq:v-21}
\begin{split}
&\frac{d}{dt}\int_{\R}\we^2 dx+2\eps\int_{\R}(\px \we)^2 dx\\
&\qquad = 2\int_{\R}\betas \we\px\we dx + 2\int_{\R}\we\psit dx\\
&\qquad \le 2\int_{\R}\vert\betas\vert\vert \we\vert\vert\px \we \vert dx + 2\int_{\R}\vert \we\vert\vert \psit\vert dx \\
&\qquad \le 2C(\tau)\int_{\R}\vert \we\vert\vert\px \we \vert dx+\int_{\R}\we^2 dx + \int_{\R} \psit^2 dx\\
&\qquad \le C(\tau)\int_{\R}\we^2 dx + C(\tau)\int_{\R}(\px \we)^2 dx + \int_{\R} \psit^2 dx.
\end{split}
\end{equation}
Differentiating with respect to $x$ the first equation of \eqref{eq:equ-v}, we get
\begin{equation}
\label{eq:vx-1}
\ptx \we =\px\betas\px \we + \betas\pxx \we + \px\psit + \gamma\px Q_{\eps} +\eps\pxxx \we.
\end{equation}
The second equation of \eqref{eq:equ-v} and \eqref{eq:vx-1} give
\begin{equation}
\label{eq:vx-2}
\ptx \we =\px\betas\px we + \betas\pxx \we + \px\psit - \gamma \we +\eps\pxxx \we
\end{equation}
Multiplying \eqref{eq:vx-2} by $\px \we$, we obtain that
\begin{align*}
\int_{\R}\ptx \we\px \we dx=&\int_{\R}\px\betas(\px \we)^2 dx +\int_{\R}\betas \pxx \we\px \we dx+ \int_{\R} \px\psit \px \we dx\\
&-\gamma\int_{\R}\we\px \we dx +\eps \int_{\R}\pxxx \we\px \we dx.
\end{align*}
Since,
\begin{align*}
\int_{\R}\ptx \we\px \we dx=&\frac{d}{dt}\left(\frac{1}{2}\int_{\R}(\px \we)^2 dx \right),\\
\int_{\R}\betas \pxx \we\px \we dx=&-\frac{1}{2}\int_{\R}\px\betas(\px \we)^2dx,\\
-\gamma\int_{\R}\we\px\we  dx&=0,
\end{align*}
due to \eqref{eq:beta2} and the Young inequality, we have that
\begin{equation}
\label{eq:v-22}
\begin{split}
&\frac{d}{dt}\int_{\R}(\px \we)^2 dx+2\eps\int_{\R}(\pxx \we)^2 dx\\
&\qquad = \int_{\R}\px\betas(\px \we)^2 dx +2 \int_{\R} \px\psit \px \we dx \\
&\qquad \le  C(\tau)\left(\frac{1}{\tau -t}+1\right)\int_{\R}(\px \we)^2 dx +2\int_{\R} \vert \px\psit\vert\vert \px \we\vert dx\\
&\qquad \le C(\tau)\left(\frac{1}{\tau -t}+1\right)\int_{\R}(\px \we)^2 dx+\int_{\R}(\px\psit)^2 dx + \int_{\R}(\px \we)^2 dx\\
&\qquad \le C(\tau)\left(\frac{1}{\tau -t}+1\right)\int_{\R}(\px \we)^2 dx+\int_{\R}(\px\psit)^2 dx.
\end{split}
\end{equation}
Adding \eqref{eq:v-21} and \eqref{eq:v-22}, we obtain that
\begin{align*}
&\frac{d}{dt}\norm{\we(t,\cdot)}^2_{H^1(\R)}+2\eps\norm{\px \we(t,\cdot)}^2_{H^1(\R)}\\
&\qquad \le C(\tau)\norm{\we(t,\cdot)}^2_{L^2(\R)}+C(\tau)\norm{\px \we(t,\cdot)}^2_{L^2(\R)}\\
&\qquad\quad +C(\tau)\left(\frac{1}{\tau -t}+1\right)\norm{\px \we(t,\cdot)}^2_{L^2(\R)}+ \norm{\psit(t,\cdot)}^2_{H^1(\R)}\\
&\qquad\le C(\tau)\norm{\we(t,\cdot)}^2_{L^2(\R)}+C(\tau)\left(\frac{1}{\tau -t}+1\right)\norm{\px \we(t,\cdot)}^2_{L^2(\R)}\\
&\qquad \quad + \norm{\psit(t,\cdot)}^2_{H^1(\R)}+C(\tau)\left(\frac{1}{\tau -t}+1\right)\norm{\we(t,\cdot)}^2_{L^2(\R)}.
\end{align*}
Therefore,
\begin{equation}
\label{eq:v-24}
\begin{split}
&\frac{d}{dt}\norm{\we(t,\cdot)}^2_{H^1(\R)}+2\eps\norm{\px \we(t,\cdot)}^2_{H^1(\R)}\\
&\qquad \le C(\tau)\left(\frac{1}{\tau -t}+1\right)\norm{\we(t,\cdot)}^2_{H^1(\R)}+\norm{\psit(t,\cdot)}^2_{H^1(\R)}.
\end{split}
\end{equation}
Let $f(t)$ be a nonnegative, absolutely continuous function on $[a,b]$, satisfying for $a.e.$ $t$ the inequality
\begin{equation*}
f'(t)+g(t) \le k(t)f(t) +h(t),
\end{equation*}
where $k(t)$, $g(t)$, $h(t)$ are nonnegative functions on $[a,b]$. Then, the Gronwall inequality says that
\begin{equation*}
f(t)+\int_{a}^{b}e^{\int_{s}^{t}k(s')ds'} g(s)ds \le e^{\int_{a}^{t}k(s)ds }\left(f(a)+\int_{a}^{t}h(s)ds\right), \quad a\le t\le b.
\end{equation*}
For \eqref{eq:v-24}, $k(t)=C(\tau)\left(\frac{1}{\tau -t}+1\right)$ and thus $e^{\int_{s}^{t}k(s')ds'}=e^{C(\tau)(t-s)}\left(\frac{\tau-s}{\tau-t}\right)^{C(\tau)}$, so we obtain, keeping in mind that $\px v(0,\cdot)=0$,
\begin{equation}
\label{eq:v-321}
\begin{split}
\norm{\we(t,\cdot)}^2_{H^1(\R)}+&2\eps \int_{0}^{t} e^{C(\tau)(t-s)}\left(\frac{\tau-s}{\tau-t}\right)^{C(\tau)}\norm{\px \we(s,\cdot)}^2_{H^1(\R)}ds\\
\le &e^{C(\tau)(t-s)}\left(\frac{\tau}{\tau-t}\right)^{C(\tau)}\int_{0}^{t}\norm{\psit(s,\cdot)}^2_{H^1(\R)}ds.
\end{split}
\end{equation}
Since $s\le t$, then $\tau-s \ge \tau -t$. Therefore,
\begin{equation}
\label{eq:v-234}
1\le \frac{\tau-s}{\tau-t}.
\end{equation}
Thus, \eqref{eq:v-321} and \eqref{eq:v-321} give \eqref{eq:v-213}.

Finally, we prove \eqref{eq:v-infty}.
Due to \eqref{eq:v-213} and the H\"older inequality, we get
\begin{align*}
\we^2(t,x)=&2\int_{-\infty}^{x}\we(t,y)\px \we(t,y) dy\le 2\int_{\R}\vert \we(t,y)\vert \vert \px \we(t,y)\vert dx\\
\le& 2\norm{\we(t,\cdot)}_{L^2(\R)}\norm{\px \we(t,\cdot)}_{L^2(\R)}\le e^{C(\tau)\tau}\left(\frac{\tau}{\tau-t}\right)^{C(\tau)}\int_{0}^{t}\norm{\psit(s,\cdot)}^2_{H^1(\R)}ds.
\end{align*}
Therefore,
\begin{equation}
\label{eq:v-325}
\vert \we(t,x)\vert \le \sqrt{2}\left(e^{C(\tau)\tau}\left(\frac{\tau}{\tau-t}\right)^{C(\tau)}\int_{0}^{t}\norm{\psit(s,\cdot)}^2_{H^1(\R)}ds\right)^{\frac{1}{2}}.
\end{equation}
\eqref{eq:v-infty} follows from \eqref{eq:v-325}.
\end{proof}
\begin{lemma}
Let $\psi_{\eps}\in C^{\infty}((0,\infty)\times\R)\cap C((0,\infty);H^{2}(\R))\cap L^{\infty}((0,\infty);H^{2}(\R))$ be a function satisfying \eqref{eq:con-psi}.
Then, using the notation introduced in \eqref{eq:v-q} and \eqref{eq:beta-psi}, there exists a function $C(\tau)>0$, independent on $\eps$ such that
\begin{equation}
\label{eq:pxv}
\norm{\px \we(t,\cdot)}_{L^{\infty}(\R)}\le C(\tau)\tau\left(\frac{\tau}{\tau-t}\right)^{C(\tau)}e^{C(\tau)\tau}.
\end{equation}
\end{lemma}
\begin{proof}
Let $p\in \N\setminus \{0\}$ be even. Thanks to \eqref{eq:vx-2},
\begin{align*}
&\frac{d}{dt}\norm{\px \we(t,\cdot)}^{p}_{L^{p}(\R)}= \frac{d}{dt}\int_{\R}(\px \we)^p dx=p\int_{\R}(\px \we)^{p-1}\ptx \we dx\\
&\qquad =p\int_{\R}\px\betas (\px \we)^p dx +p \int_{\R}\betas\pxx \we (\px \we)^{p-1} dx +p\int_{\R}\px\psit(\px \we)^{p-1} dx \\
&\qquad\quad -p\gamma\int_{\R}\we(\px \we)^{p-1} dx +\eps p\int_{\R}\pxxx \we (\px \we)^{p-1} dx\\
&\qquad =(p-1)\int_{\R}\px\betas (\px \we)^p dx +p\int_{\R}\px\psit(\px \we)^{p-1} dx-p\gamma \int_{\R}\we(\px \we)^{p-1} dx\\
&\qquad\quad -p(p-1)\eps\int_{\R}(\px \we)^{p-2}(\pxx \we)^2 dx,
\end{align*}
that is,
\begin{align*}
\frac{d}{dt}\norm{\px \we(t,\cdot)}^{p}_{L^{p}(\R)}\le &(p-1)\int_{\R}\px\betas (\px \we)^p dx\\
&+p\int_{\R}\px\psit(\px \we)^{p-1} dx-p\gamma \int_{\R}\we(\px \we)^{p-1} dx.
\end{align*}
Due to \eqref{eq:beta2} and the H\"older inequality, we get
\begin{align*}
\frac{d}{dt}\norm{\px \we(t,\cdot)}^{p}_{L^{p}(\R)}\le &(p-1)C(\tau)\left(\frac{1}{\tau-t}+1\right)\norm{\px \we(t,\cdot)}^{p}_{L^{p}(\R)}\\
&+p\int_{\R}\vert\px\psit\vert\vert(\px \we)^{p-1}\vert dx+p\gamma \int_{\R}\vert \we\vert\vert(\px \we)^{p-1}\vert dx \\
\le& (p-1)C(\tau)\left(\frac{1}{\tau-t}+1\right)\norm{\px \we(t,\cdot)}^{p}_{L^{p}(\R)}\\
&+ p\norm{\px\psit(t,\cdot)}_{L^{p}(\R)}\norm{(\px \we(t,\cdot))^{p-1}}_{L^{\frac{p}{p-1}}(\R)}\\
&+ p\gamma \norm{\we(t,\cdot)}_{L^{p}(\R)}\norm{(\px \we(t,\cdot))^{p-1}}_{L^{\frac{p}{p-1}}(\R)}.
\end{align*}
Since
\begin{equation*}
\norm{\px\psit(t,\cdot)}_{L^{p}(\R)}\le \alpha_{1},
\end{equation*}
where $\alpha_{1}$ is a positive constant which does not depend on $\eps$, we have that
\begin{equation}
\label{eq:v32}
\begin{split}
\frac{d}{dt}\norm{\px \we(t,\cdot)}^{p}_{L^{p}(\R)}\le & (p-1)C(\tau)\left(\frac{1}{\tau-t}+1\right)\norm{\px \we(t,\cdot)}^{p}_{L^{p}(\R)}\\
&+p\alpha_{1}\norm{\px \we(t,\cdot)}^{p-1}_{L^{p}(\R)}\\
&+p\gamma \norm{\we(t,\cdot)}_{L^{p}(\R)}\norm{\px \we(t,\cdot)}^{p-1}_{L^{p}(\R)}.
\end{split}
\end{equation}
Hence,
\begin{align*}
p\norm{\px \we(t,\cdot)}^{p-1}_{L^{p}(\R)}&\frac{d}{dt} \norm{\px \we(t,\cdot)}_{L^{p}(\R)}\\
\le & (p-1)C(\tau)\left(\frac{1}{\tau-t}+1\right)\norm{\px \we(t,\cdot)}^{p}_{L^{p}(\R)}\\
& + p\alpha_{1}\norm{\px \we(t,\cdot)}^{p-1}_{L^{p}(\R)}+p\gamma \norm{\we(t,\cdot)}_{L^{p}(\R)}\norm{\px \we(t,\cdot)}^{p-1}_{L^{p}(\R)},
\end{align*}
that is
\begin{equation}
\begin{split}
\label{eq:v34}
\frac{d}{dt} \norm{\px \we(t,\cdot)}_{L^{p}(\R)}\le &\frac{(p-1)}{p}C(\tau)\left(\frac{1}{\tau-t}+1\right)\norm{\px \we(t,\cdot)}_{L^{p}(\R)}\\
&+\alpha_{1} + \gamma \norm{\we(t,\cdot)}_{L^{p}(\R)}.
\end{split}
\end{equation}
Due to \eqref{eq:v-213} and \eqref{eq:v-infty},
\begin{align*}
\int_{\R}\vert \we(t,x) \vert ^p dx=& \int_{\R}\vert \we(t,x)\vert^{p-2}\we^2(t,x)dx\\
\le&\norm{\we(t,\cdot)}^{p-2}_{L^{\infty}(\R)}\norm{\we(t,\cdot)}^2_{L^2(\R)}\le (F_{1}(\tau,t,\psit))^{p-2}(F_{2}(\tau,t,\psit)),
\end{align*}
where
\begin{equation}
\label{eq:v45}
\begin{split}
F_{1}(\tau,t,\psit)&=\sqrt{2}\left( e^{C(\tau)\tau}\left(\frac{\tau}{\tau-t}\right)^{C(\tau)}\int_{0}^{t}\norm{\psit(s,\cdot)}^2_{H^1(\R)}ds\right)^{\frac{1}{2}},\\
F_{2}(\tau,t,\psit)&=e^{C(\tau)\tau}\left(\frac{\tau}{\tau-t}\right)^{C(\tau)}\int_{0}^{t}\norm{\psit(s,\cdot)}^2_{H^1(\R)}ds.
\end{split}
\end{equation}
Thus,
\begin{equation}
\label{eq:v40}
\norm{\we(t,\cdot)}_{L^{p}(\R)}\le (F_{1}(\tau,t,\psit))^{\frac{p-2}{p}}(F_{2}(\tau,t,\psit))^{\frac{1}{p}}.
\end{equation}
Hence, \eqref{eq:v34} and \eqref{eq:v40} give
\begin{equation}
\begin{split}
\frac{d}{dt} \norm{\px \we(t,\cdot)}_{L^{p}(\R)}\le &\frac{(p-1)}{p}C(\tau)\left(\frac{1}{\tau-t}+1\right)\norm{\px \we(t,\cdot)}_{L^{p}(\R)}\\
&+\alpha_{1} + \gamma (F_{1}(\tau,t,\psit))^\frac{p-2}{p}(F_{2}(\tau,t,\psit))^{\frac{1}{p}}.
\end{split}
\end{equation}
Keeping in mind that $\px v(0,\cdot)=0$, the Gronwall Lemma gives
\begin{align*}
&\norm{\px \we(t,\cdot)}_{L^{p}(\R)}\le \alpha_{1} e^{\frac{p-1}{p}C(\tau)\left(\log\left(\frac{\tau}{\tau-t}\right)+t\right)}\int_{0}^{t} e^{-\frac{p-1}{p}C(\tau)\left(\log\left(\frac{\tau}{\tau-s}\right)+s\right)}ds\\
&\qquad+\gamma e^{\frac{p-1}{p}C(\tau)\left(\log\left(\frac{\tau}{\tau-t}\right)+t\right)}\int_{0}^{t}e^{-\frac{p-1}{p}C(\tau)\left(\log\left(\frac{\tau}{\tau-s}\right)+s\right)}\\
&\qquad\quad\cdot (F_{1}(\tau,s,\psit))^\frac{p-2}{p}(F_{2}(\tau,s,\psit))^{\frac{1}{p}}ds\\
&\quad \le \alpha_{1} \left(\frac{\tau}{\tau-t}\right)^{C(\tau)\frac{p-1}{p}}e^{\frac{p-1}{p}C(\tau)t}\int_{0}^{t}e^{-\frac{p-1}{p}C(\tau)s}\left(\frac{\tau-s}{\tau}\right)^{\frac{p-1}{p}C(\tau)}ds\\     &\qquad +\gamma\left(\frac{\tau}{\tau-t}\right)^{C(\tau)\frac{p-1}{p}}e^{\frac{p-1}{p}C(\tau)t}\int_{0}^{t}e^{-\frac{p-1}{p}C(\tau)s}\left(\frac{\tau-s}{\tau}\right)^{\frac{p-1}{p}C(\tau)}\\
&\qquad\quad\cdot (F_{1}(\tau,s,\psit))^\frac{p-2}{p}(F_{2}(\tau,s,\psit))^{\frac{1}{p}}ds\\
&\quad \le \alpha_{1} \left(\frac{\tau}{\tau-t}\right)^{C(\tau)\frac{p-1}{p}}e^{\frac{p-1}{p}C(\tau)\tau}\tau+\gamma\left(\frac{\tau}{\tau-t}\right)^{C(\tau)\frac{p-1}{p}}e^{\frac{p-1}{p}C(\tau)\tau}\\
&\qquad \quad\cdot \int_{0}^{t}\left(\frac{\tau-s}{\tau}\right)^{\frac{p-1}{p}C(\tau)} (F_{1}(\tau,s,\psit))^{\frac{p-2}{p}}(F_{2}(\tau,s,\psit))^{\frac{1}{p}}ds.
\end{align*}
We observe that, for \eqref{eq:v45},
\begin{align*}
&\int_{0}^{t}\left(\frac{\tau-s}{\tau}\right)^{\frac{p-1}{p}C(\tau)}(F_{1}(\tau,s,\psit))^{\frac{p-2}{p}}(F_{2}(\tau,s,\psit))^{\frac{1}{p}}ds\\
&\quad \le 2^{\frac{p-2}{2p}}\alpha^{\frac{1}{2}}_{2}\int_{0}^{t}e^{C(\tau)\tau\frac{p-2}{2p}}e^{C(\tau)\tau\frac{1}{p}}\left(\frac{\tau-s}{\tau}\right)^{\frac{p-1}{p}C(\tau)}\left(\frac{\tau}{\tau-s}\right)^{C(\tau)\frac{p-2}{2p}}\\
&\qquad\quad\cdot\left(\frac{\tau}{\tau-s}\right)^{C(\tau)\frac{1}{p}}s^{\frac{1}{2}}ds\\
&\quad \le 2^{\frac{p-2}{2p}}\alpha^{\frac{1}{2}}_{2}t^{\frac{1}{2}}e^{\frac{C(\tau)\tau}{2}}\int_{0}^{t}\left(\frac{\tau-s}{\tau}\right)^{\frac{p-1}{p}C(\tau)}\left(\frac{\tau-s}{\tau}\right)^{-C(\tau)\frac{p-2}{2p}}\\
&\qquad\quad\cdot\left(\frac{\tau-s}{\tau}\right)^{-C(\tau)\frac{1}{p}}ds\\
&\quad \le 2^{\frac{p-2}{2p}}\alpha^{\frac{1}{2}}_{2}\tau^{\frac{1}{2}}e^{\frac{C(\tau)\tau}{2}}\int_{0}^{t}\left(\frac{\tau-s}{\tau}\right)^{\frac{p-2}{2p}C(\tau)}ds\\
&\quad\le 2^{\frac{p-2}{2p}}\alpha^{\frac{1}{2}}_{2}\tau^{\frac{1}{2}}e^{\frac{C(\tau)\tau}{2}}\int_{0}^{t} ds\le 2^{\frac{p-2}{2p}}\alpha^{\frac{1}{2}}_{2}\tau^{\frac{3}{2}}e^{\frac{C(\tau)\tau}{2}},
\end{align*}
where $\alpha_{2}$ is a positive constant independent on $\eps$ such that
\begin{equation*}
\norm{\psit(s,\cdot)}^2_{H^1(\R)}\le \alpha_{2}.
\end{equation*}
Hence,
\begin{align*}
\norm{\px \we(t,\cdot)}_{L^{p}(\R)}\le&\alpha_{1} \left(\frac{\tau}{\tau-t}\right)^{C(\tau)\frac{p-1}{p}}e^{\frac{p-1}{p}C(\tau)\tau}\tau\\
&+\gamma2^{\frac{p-2}{2p}}\alpha^{\frac{1}{2}}_{2}\tau^{\frac{3}{2}}\left(\frac{\tau}{\tau-t}\right)^{C(\tau)\frac{p-1}{p}}e^{\frac{p-1}{p}C(\tau)\tau}e^{\frac{C(\tau)\tau}{2}}\\
\le&\left(\alpha_{1}+\gamma\alpha_{2}^{\frac{1}{2}}2^{\frac{p-2}{2p}}\tau^{\frac{1}{2}}e^{\frac{C(\tau)\tau}{2}}\right) \tau\left(\frac{\tau}{\tau-t}\right)^{C(\tau)\frac{p-1}{p}}e^{\frac{p-1}{p}C(\tau)\tau}.
\end{align*}
Sending $p\to\infty$, we have
\begin{align*}
\norm{\px \we(t,\cdot)}_{L^{p}(\R)}\le& \left(\alpha_{1}+2\gamma\alpha_{2}^{\frac{1}{2}}\tau^{\frac{1}{2}}e^{\frac{C(\tau)\tau}{2}}\right)\tau\left(\frac{\tau}{\tau-t}\right)^{C(\tau)}e^{C(\tau)\tau}\\
\le& C(\tau)\tau\left(\frac{\tau}{\tau-t}\right)^{C(\tau)}e^{C(\tau)\tau},
\end{align*}
which gives \eqref{eq:pxv}.
\end{proof}
Coming back to the terminal value problem, the previous results for the initial value problem translate into the following ones for \eqref{eq:test-funct-10}:
\begin{corollary}\label{col:1}
Let $\psi_{\eps}\in C^{\infty}((0,\infty)\times\R)\cap C((0,\infty);H^{2}(\R))\cap L^{\infty}((0,\infty);H^{2}(\R))$ be a function satisfying \eqref{eq:con-psi}.
Then for each $\eps>0$ and $t\in (0,\tau)$
\begin{equation}
\begin{split}
\label{eq:v54}
\norm{\phi_{\eps}(t,\cdot)}^2_{H^1(\R)}&+2\eps \int_{t}^{\tau}\norm{\px \phi_{\eps}(s,\cdot)}^2_{H^1(\R)}ds\\\le &e^{C(\tau)\tau}\left(\frac{\tau}{t}\right)^{C(\tau)}\int_{t}^{\tau}\norm{\psi_{\eps}(s,\cdot)}^2_{H^1(\R)}ds,
\end{split}
\end{equation}
\begin{align}
\label{eq:v55}
\norm{\phi_{\eps}(t,\cdot)}_{L^{\infty}(\R)}&\le\sqrt{2}\left( e^{C(\tau)\tau}\left(\frac{\tau}{t}\right)^{C(\tau)}\int_{t}^{\tau}\norm{\psi_{\eps}(s,\cdot)}^2_{H^1(\R)}ds\right)^{\frac{1}{2}},\\
\label{eq:v56}
\norm{\px \phi_{\eps}(t,\cdot)}_{L^{\infty}(\R)}&\le C(\tau)\tau\left(\frac{\tau}{t}\right)^{C(\tau)}e^{C(\tau)\tau}.
\end{align}
\end{corollary}
Although we will not use this fact directly, an interesting consequence of the previous estimates is the existence of a solution of \eqref{eq:test-funct-1}
\begin{theorem}
Fix any $0<\delta<\tau$. Then there exists at least one distributional solution $(\phi,\Phi)\in L^{\infty}((\delta,\tau);W^{1,\infty}(\R)\cap H^{1}(\R))\times L^{\infty}((\delta,\tau);W^{2,\infty}(\R)\cap H^{2}(\R))$ to the terminal value problem \eqref{eq:test-funct-1}.
\end{theorem}
\begin{proof}
For each fixed $\eps>0$, let $(\phi,\Phi)$ denote the solution of \eqref{eq:test-funct-10}. Due to the second equation of \eqref{eq:test-funct-10} and Corollary \ref{col:1},
\begin{equation}
\label{eq:v201}
\begin{split}
&\{\phi_{\eps}\}_{\eps>0} \quad \textrm{is bounded in $L^{\infty}((\delta,\tau);W^{1,\infty}(\R)\cap H^{1}(\R))$}, \quad \textrm{for $\delta\in (0,\tau)$},\\
&\{\Phi_{\eps}\}_{\eps>0} \quad \textrm{is bounded in $L^{\infty}((\delta,\tau);W^{2,\infty}(\R)\cap H^{2}(\R))$}, \quad \textrm{for $\delta\in (0,\tau)$}.
\end{split}
\end{equation}
Then, there exist
\begin{equation*}
\phi\in L^{\infty}((\delta,\tau);W^{1,\infty}(\R)\cap H^{1}(\R)), \quad \Phi\in L^{\infty}((\delta,\tau);W^{2,\infty}(\R)\cap H^{2}(\R)), \quad 0<\delta<\tau,   \end{equation*}
and $\{\eps_{k}\}_{k\in\N}$, $\eps_{k}\to 0$, such that
\begin{equation}
\label{eq:v207}
\begin{split}
&\phi_{\eps_{k}}\rightharpoonup\phi \quad \textrm{weakly in $L^{p}((\delta,\tau);W^{1,q}(\R))$}, \quad  \textrm{for $\delta\in (0,\tau),\, 1\le p <\infty, \ 2\le q <\infty$,}\\
&\Phi_{\eps_{k}}\rightharpoonup\Phi \quad \textrm{weakly in $L^{p}((\delta,\tau);W^{2,q}(\R))$}, \quad  \textrm{for $\delta\in (0,\tau),\, 1\le p <\infty,\, 2\le q <\infty$.}
\end{split}
\end{equation}
It remains to verify that the limit pair $(\phi,\Phi)$ is a solution of \eqref{eq:test-funct-1} in the sense of distributions. Fix any $\phi\in C^{\infty}_{c}((0,\tau)\times\R)$. We need to show that
\begin{equation}
\label{eq:v142}
\int_{0}^{\tau}\!\!\!\!\int_{\R}\phi b_{\eps_{k}}\px\phi_{\eps_{k}} dt dx \to \int_{0}^{\tau}\!\!\!\!\int_{\R}\phi b\px\phi dt dx.
\end{equation}
Observe that
\begin{equation}
\label{eq:v202}
\begin{split}
\int_{0}^{\tau}\!\!\!\!\int_{\R}\phi(b_{\eps_{k}}\px\phi_{\eps_{k}}-b\px\phi) dt dx=&\int_{0}^{\tau}\!\!\!\!\int_{\R}\phi\left(b_{\eps_{k}}-b\right)\px\phi_{\eps_{k}}dtdx\\
&+\int_{0}^{\tau}\!\!\!\!\int_{\R}\phi b\left(\px\phi_{\eps_{k}}-\px\phi\right)dtdx.
\end{split}
\end{equation}
Since $\phi$ has compact support in $(0,\tau)\times\R$, there exists $\delta>0$ such that
$\text{supp}(\phi)\subset (\delta,\tau)\times\R$. Therefore,
we can employ \eqref{eq:b1} and \eqref{eq:v201} to obtain
\begin{equation}
\label{eq:v203}
\begin{split}
&\int_{0}^{\tau}\!\!\!\!\int_{\R}\phi\left(b_{\eps_{k}}-b\right)\px\psi_{\eps_{k}}dtdx=\int_{\delta}^{\tau}\!\!\!\!\int_{\R}\phi\left(b_{\eps_{k}}-b\right)\px\psi_{\eps_{k}}dtdx\\
&\qquad \le \norm{b_{\eps_{k}}-b}_{L^{2}((0,\tau)\times\R)}\norm{\phi}_{L^\infty((0,\tau)\times\R)}\norm{\px\phi_{\eps_{k}}}_{L^{2}((\delta,\tau)\times\R)}\to 0.
\end{split}
\end{equation}
Since $\phi b\in L^2((0,\tau)\times\R)$ and $\text{supp}(\phi b)\subset (\delta,\tau)\times\R$, it follows from \eqref{eq:v207} that
\begin{equation}
\label{eq:205}
\int_{0}^{\tau}\!\!\!\!\int_{\R}\phi b\left(\px\phi_{\eps_{k}}-\px\phi\right)dtdx\to 0.
\end{equation}
\eqref{eq:v202}, \eqref{eq:v203} and \eqref{eq:205} give \eqref{eq:v142}, and the proof is completed.
\end{proof}

Now, we are ready for the proof of Theorem \ref{th:main}.
\begin{proof}[Proof of Theorem \ref{th:main}]
In Section \ref{sec:Es}, it has been proved the existence of an entropy solution of \eqref{eq:OHw-u}, or \eqref{eq:OHw}. Moreover, for  Theorem \ref{th:olei}, we have that $i)$ implies $ii)$.

Let us show that $ii)$ implies $i)$. It is sufficient to prove that there exists an unique weak solution of \eqref{eq:OHw-u}, or \eqref{eq:OHw}, that verifies
\eqref{eq:ole1}.
Let us suppose that $u$ and $v$ are two weak solution of \eqref{eq:OHw-u}, or \eqref{eq:OHw}. We have to prove that \eqref{eq:u-uguale-v-1} holds true.

We begin by fixed a test function $\psi \in C^{\infty}_{c}((0,\infty)\times\R)$. Let $0< \tau_{0} <\tau_{1}$ be such that
\begin{equation}
\label{eq:psi-1}
\supp(\psi)\subset (\tau_{0},\tau_{1})\times\R.
\end{equation}
From Lemma \ref{lm:dual-proble}, for each $\eps>0$ there exists a unique $\widetilde {\phi_{\eps}}\in  C^{\infty}((0,\infty)\times\R)\cap C((0,\infty);H^{\ell}(\R))$,  $\ell>2$, solving \eqref{eq:test-funct-10}. Let $\{\phi_{\eps}\}_{\eps}\subset C^{\infty}_{c}((0,\tau_{1})\times\R)$ be such that
\begin{equation}
\label{eq:supp-phi}
\eps\left\vert\supp(\phi_{\eps})\right\vert\to 0,
\end{equation}
\begin{equation}
\label{eq:v125}
\widetilde{\phi_{\eps}}-\phi_{\eps}\to 0 \quad \textrm{strongly in } \quad \begin{cases}
&L^{1}((0,\infty);W^{2,1}(\R))\cap W^{1,1}((0,\infty)\times\R)\\
&\quad \cap W^{1,\infty}((0,\infty);H^{1}(\R))\cap L^{\infty}((0,\infty);H^{\ell}(\R)),
\end{cases}
\end{equation}
with $\ell>2$, and define the family $\{\psi_{\eps}\}_{\eps}$ as follows
\begin{equation}
\label{eq:def-di-psi}
\psi_{\eps}=\pt\phi_{\eps}+\be\px\phi_{\eps}+\gamma\Phi_{\eps}+\eps\pxx\phi_{\eps}, \quad \eps>0.
\end{equation}
Clearly,
\begin{equation*}
\psi_{\eps}\in C^{\infty}((0,\infty)\times\R)\cap C((0,\infty);H^2(\R)) \quad \eps>0,
\end{equation*}
and, due to \eqref{eq:b1}, \eqref{eq:b2} and \eqref{eq:v125},
\begin{equation}
\label{eq:psi-12}
\psi_{\eps}\to \psi  \quad \textrm{strongly in} \quad L^{1}((0,\infty)\times\R)\cap L^{\infty}((0,\infty);H^{2}(\R)).
\end{equation}
In particular, $\phi_{\eps}$ and $\psi_{\eps}$ satisfy the two equations (see \eqref{eq:test-funct-10} and \eqref{eq:def-di-psi})
\begin{equation}
\label{eq:v126}
\pt\phi_{\eps}+\be\px\phi_{\eps}+\gamma\Phi_{\eps}=\psi_{\eps}-\eps\pxx\phi_{\eps}, \quad \px\Phi_{\eps}= -\phi_{\eps}.
\end{equation}
Hence, using \eqref{eq:psi-1} and \eqref{eq:v126},
\begin{equation}
\begin{split}
\label{eq:v127}
\int_{0}^{\infty}\!\!\!\int_{\R}\omega\psi dtdx=&\int_{\tau_{0}}^{\tau_{1}}\!\!\!\int_{\R} \omega\psi dtdx\\
=&\int_{\tau_{0}}^{\tau_{1}}\!\!\!\int_{\R} \omega\psi_{\eps} dtdx + \int_{\tau_{0}}^{\tau_{1}}\!\!\!\int_{\R} \omega(\psi -\psi_{\eps}) dtdx\\
=&\int_{\tau_{0}}^{\tau_{1}}\!\!\!\int_{\R} \omega\left(\pt\phi_{\eps}+\be\px\phi_{\eps}+\gamma\Phi_{\eps}+\eps\pxx\phi_{\eps}\right)dtdx\\
&+ \int_{\tau_{0}}^{\tau_{1}}\!\!\!\int_{\R} \omega(\psi -\psi_{\eps}) dtdx\\
=& \int_{\tau_{0}}^{\tau_{1}}\!\!\!\int_{\R} \omega\left(\pt\phi_{\eps}+b\px\phi_{\eps}+\gamma\Phi_{\eps}\right)dtdx\\
&+\eps\int_{\tau_{0}}^{\tau_{1}}\!\!\!\int_{\R} \omega\pxx\phi_{\eps}dtdx + \int_{\tau_{0}}^{\tau_{1}}\!\!\!\int_{\R} \omega\left(\be - b\right)\px\phi_{\eps} dtdx\\
&+\int_{\tau_{0}}^{\tau_{1}}\!\!\!\int_{\R} \omega(\psi -\psi_{\eps}) dtdx.
\end{split}
\end{equation}
Using the fact that $\phi_{\eps} \in C^{\infty}_{c}((0,\infty)\times\R)$ and \eqref{eq:test-fu-3}, we have that
\begin{equation}
\label{eq:v121}
\int_{\tau_{0}}^{\tau_{1}}\!\!\!\int_{\R} \omega\left(\pt\phi_{\eps}+b\px\phi_{\eps}+\gamma\Phi_{\eps}\right)dtdx=0.
\end{equation}
Employing \eqref{eq:v1230}, \eqref{eq:v54}, \eqref{eq:supp-phi} and the H\"older inequality
\begin{equation}
\label{eq:v215}
\begin{split}
\left\vert \eps \int_{\tau_{0}}^{\tau_{1}}\!\!\!\int_{\R}\omega\pxx\phi dtdx\right\vert \le& \eps \norm{\omega}_{L^{\infty}((\tau_{0},\tau_{1})\times\R)}\norm{\pxx\phi_{\eps}}_{L^{1}((\tau_{0},\tau_{1})\times\R)}\\
\le & \eps \norm{\omega}_{L^{\infty}((\tau_{0},\tau_{1})\times\R)}\sqrt{\left\vert\supp(\phi_{\eps})\right\vert} \norm{\pxx\phi_{\eps}}_{L^{2}((\tau_{0},\tau_{1})\times\R)}\\
\le & \left(\eps\frac{\vert\supp(\phi_{\eps})\vert}{2}\right)^{\frac{1}{2}}\norm{\omega}_{L^{\infty}((\tau_{0},\tau_{1})\times\R)}\\
&\cdot e^{\frac{C(\tau)\tau}{2}}\left(\frac{\tau_{1}}{\tau_{0}}\right)^{\frac{C(\tau)}{2}}\left(\int_{\tau_{0}}^{\tau_{1}}\norm{\psi_{\eps}(s,\cdot)}^2_{H^1(\R)}ds\right)^{\frac{1}{2}}\to 0.
\end{split}
\end{equation}
It follows from \eqref{eq:v1230}, \eqref{eq:b1}, \eqref{eq:v54} and the H\"older inequality that
\begin{equation}
\label{eq:v216}
\begin{split}
&\left\vert \int_{\tau_{0}}^{\tau_{1}}\!\!\!\int_{\R}\omega\left(b - \be\right)\px\phi_{\eps}dtdx\right\vert\\
&\quad \le\norm{\omega}_{L^{\infty}((\tau_{0},\tau_{1})\times\R)}\norm{b-\be}_{L^{2}((\tau_{1},\tau_{0})\times\R)}\norm{\phi_{\eps}}_{L^{2}((\tau_{1},\tau_{0})\times\R)}\\
&\quad \le \norm{\omega}_{L^{\infty}((\tau_{0},\tau_{1})\times\R)}\sqrt{\tau_{1}-\tau_{0}}\norm{b-\be}_{L^{2}((\tau_{1},\tau_{0})\times\R)}\\
&\qquad \cdot\left(\eps\frac{\vert\supp(\phi_{\eps})\vert}{2}\right)^{\frac{1}{2}}\norm{\omega}_{L^{\infty}((\tau_{0},\tau_{1})\times\R)}e^{\frac{C(\tau)\tau}{2}}\\
&\qquad\cdot\left(\frac{\tau_{1}}{\tau_{0}}\right)^{\frac{C(\tau)}{2}}\left(\int_{\tau_{0}}^{\tau_{1}}\norm{\psi_{\eps}(s,\cdot)}^2_{H^1(\R)}ds\right)^{\frac{1}{2}}\to 0.
\end{split}
\end{equation}
Due to \eqref{eq:v1230} and \eqref{eq:psi-12}, we get
\begin{equation}
\label{eq:v255}
\left\vert \int_{\tau_{0}}^{\tau_{1}}\!\!\!\int_{\R}\omega\left(\psi-\psi_{\eps}\right)dtdx\right\vert \le \norm{\omega}_{L^{\infty}((\tau_{0},\tau_{1})\times\R)}\norm{\psi-\psi_{\eps}}_{L^{1}((0,\infty)\times\R)}\to 0.
\end{equation}
Summarizing, using \eqref{eq:v121}, \eqref{eq:v215}, \eqref{eq:v216} and \eqref{eq:v255} in \eqref{eq:v127}  yields
\begin{equation*}
\int_{\tau_{0}}^{\tau_{1}}\!\!\!\int_{\R}\omega\psi dtdx =0.
\end{equation*}
Due to the freedom in the choice of $\psi$, this implies \eqref{eq:u-uguale-v-1}, and the proof is completed.
\end{proof}

\end{document}